\documentclass{article}

\usepackage{titling}
\usepackage{amsmath,amssymb,amsthm}
\usepackage{physymb}

\usepackage[all]{xy}
\usepackage{multicol}

\usepackage{enumerate}
\DeclareMathOperator{\attr}{Attr}
\newcommand{\bidir}{%
  \twoheadleftarrow
  \mathrel{\mkern-15mu}%
  \twoheadrightarrow
}
\newcommand{\ced}{\mathcal{Y}}
\newcommand{\htap}{\twoheadleftarrow}
\newcommand{\ed}{\mathcal{E}}
\DeclareMathOperator{\End}{{End}}
\newcommand{\ext}{\mathrm{P}}
\newcommand{\Id}{\mathrm{Id}}
\newcommand{\id}{\mathrm{id}}

\newcommand{\fcal}{\mathcal{F}}
\newcommand{\loops}{\mathcal{L}}

\newcommand{\ned}{\mathcal{N}}
\newcommand{\path}{\twoheadrightarrow}
\newcommand{\paths}{\mathcal{P}}
\newcommand{\rel}{\mathbf{r} }
\newcommand{\relc}{\mathbf{c}}
\newcommand{\res}{\mathrm{Res}}
\newcommand{\rs}[1]{\mathcal{H}(\Sigma_{#1})}

\newcommand{\zed}{\mathcal{Z}}
\newcommand{\zigzag}{\rightsquigarrow}

\newcommand{\N}{\mathbb{N}}

\newcommand{\cat}{\mathcal{C}}

\newcommand{\hecke}{\mathfrak{H}}
\newcommand{\sym}{\Sigma}

\newtheorem{theorem}{Theorem}[section]

\newtheorem{proposition}[theorem]{Proposition}
\newtheorem{corollary}[theorem]{Corollary}

\theoremstyle{definition}
\newtheorem{definition}[theorem]{Definition}
\theoremstyle{remark}
\newtheorem{remark}[theorem]{Remark}
\newtheorem{example}[theorem]{Example}

\newcommand{\si}[1]{\scriptsize{\emph{#1}}}

\preauthor{}
\postauthor{}
\DeclareRobustCommand{\authorthing}{
\begin{center}
	\begin{tabular*}{0.75\textwidth}{@{\extracolsep{\fill}}|lc}
		 \quad  Ivan Yudin \thanks{This work was supported by the Centre for Mathematics of the 
University of Coimbra -- UID/MAT/00324/2013, funded by the Portuguese 
Government through FCT/MEC and co-funded by the European Regional 
Development Fund through the Partnership Agreement PT2020, and 
		 by  Programa
Investigador FCT IF/00016/2013 funded by Portuguese state budget.}\\	
 \si{\quad yudin@mat.uc.pt}\\
{ \si{CMUC, Department of Mathematics, University of Coimbra,
Coimbra, Portugal} }\\
	\end{tabular*}
\end{center}
}

\date{}
\author{\authorthing}

\title{Decreasing diagrams and coherent presentations}

\begin{document}
\maketitle

\begin{abstract}
We show how decreasing diagrams introduced in the theory of rewriting systems
can be used to prove coherence type theorems in category theory. We apply
this method to describe a coherent presentation of the $0$-Hecke monoid
$\hecke(\sym_n)$ of the symmetric
group $\sym_n$, i.e. a presentation by generators, relations, and relations
between relations.
\end{abstract}

\section{Introduction}
%Given a monoid $M$ and a set $S$, an action of $M$ on $S$ is determined by a
%collection of operators $A_m$, $m\in M$, on $S$ such that $A_m A_n =
%A_{mn}$. In many cases it is difficult to give an explicit form of $A_m$ for all
%$m\in M$, In fact, if $X\subset M$ generates $M$, then it is enough to specify
%the operators $A_x$, $x\in X$. Then there is a set of relations $\rel$ between
%elements in $X$, such that the collection $\left\{A_x\middle|x\in X\right\}$ determines an action
%of $M$ on $S$ if and only if it satisfies the operators $A_x$ satisfy these
%relations. The pair $\left\langle X,\rel \right\rangle$ is a presentation of
%$M$. 

Presentations of monoids are known to be essential in the study of actions on sets.
 In fact, given a monoid $M$ with
presentation $\left\langle X,\rel \right\rangle$, to describe an action of $M$ on a set
$S$ 
  it is enough to give a set of endomorphisms $A_x\colon S\to S$, $x\in X$,
satisfying the
relations in $\rel$.
% As an example of  utility of an explicit
%presentation, we mention the
%celebrated work of Young which describes irreducible representations of the symmetric
%group over a field of characteristic zero. 

In recent years the interest for actions of monoids on categories raised. 
In this case presentations have to be replaced by the so called \emph{coherent
presentations}. 
%
%
%
%On the other hand to study actions of a monoid $M$ on a category $\cat$, we need
%to extend the presentation $\left\langle X,\rel \right\rangle$ of $M$ to a so
%called \emph{coherent presentation} $\left\langle X,\rel,\ed \right\rangle$ of
%$M$. 
Suppose we are given 
 an action of $M$ on  $\cat$. Then we have a collection
of endofunctors $F_m\colon \cat \to \cat$, $m\in M$, for which  $F_m \circ F_n = F_{mn}$ does not hold in general. Instead, we have a collection of natural
isomorphisms $\lambda_{m,n}\colon F_m \circ F_n \to F_{mn}$ such that the
natural transformations
\begin{equation}
\label{equal}
\begin{gathered}
F_{l}\circ F_{m} \circ F_n \xrightarrow{\lambda_{l,m}\circ F_n} F_{lm} \circ F_n
\xrightarrow{\lambda_{lk,n}}F_{lmn}\\
F_{l}\circ F_{m} \circ F_n \xrightarrow{F_l\circ \lambda_{m,n}} F_{l} \circ
F_{mn} \xrightarrow{\lambda_{l,kn}}F_{lmn}
\end{gathered}
\end{equation}
are equal for all $l$, $m$, $n\in M$. Further, using
$\left\{\,\lambda_{m,n}\,\middle|\, m,n\in M\,\right\}$, we can construct
natural isomorphisms
\begin{equation*}
\lambda_{m_1,\dots, m_k} \colon F_{m_1}\circ \dots \circ F_{m_k} \to
F_{m_1\dots m_k}
\end{equation*}
for all $m_1$,\dots, $m_k\in M$. 
 Now for  every relation
\begin{equation*}
r = (x_{1} \dots x_{s} , y_{1}\dots y_{t})
\end{equation*}
in $\rel$ we define the natural isomorphism $\tau_r = \lambda_{y_1,\dots,
y_t}^{-1} \lambda_{x_1,\dots,x_s}$. It can be shown that one can reconstruct
(up
to an isomorphism) the action 
\begin{equation*}
\left\{\, F_m,\ \lambda_{m,n} \,\middle|\, m,n\in
M \right\} 
\end{equation*}
of $M$ on $\cat$ from the collection $\left\{\, F_x,\ \tau_r \,\middle|\,
x\in x,\ \tau_r \right\}$. 

Not every collection $\left\{\, F_x,\ \tau_r \,\middle|\, x\in X,\ r\in \rel
\right\}$ of endofunctors of $\cat$ and natural transformations between them can
be obtained from an action of $M$ on $\cat$. The obstruction comes from the
axiom~\eqref{equal}. This obstruction can be translated into a set $\ed$  of equations
of the form
\begin{equation*}
\tau_{r_1} \dots \tau_{r_k} = \tau_{s_1}\dots \tau_{s_l}
\end{equation*}
with $r_1$, \dots $r_k$, $s_1$,\dots, $s_l\in \rel$. Such a  set $\ed$ is called
a complete
set of relations between relations from $\rel$, and $\left\langle X,\rel,\ed
\right\rangle$ is called a
\emph{coherent presentation} of $M$. 

%
%If we are interested in actions of $
%M$ on a category $\cat$ 
%%(representations of $M$ in the $2$-category
%%$\mathrm{Cat}$ of all categories),
% then it is not enough to have a set of  generators
%$X$ of $M$ and
%a set of relations $\rel$ between them to describe it. One needs an additional information.
%Namely, to the generators of $M$ one puts in the correspondence endofunctors
%$F_x$, to the relations $r\in \rel$ correspond certain natural isomorphisms
%between compositions of functors, but one also needs that these natural
%isomorphisms satisfy some set of equations, in order this data determines an
%action of $M$ on $\cat$. The triple $\left\langle X,\rel, \ed \right\rangle$ of
%generators $X$, relations $\rel$, and  \emph{relations between relations}
%$\ed$ is called a coherent presentation of the  monoid. 

To find coherent presentations of some monoids,
Guiraud and Malbos~\cite{plactic} used higher dimensional rewriting theory.
Applying the same technique, Gausent, Guiraud, and Malbos obtained a coherent
presentation of Artin braid groups associated with Coxeter
systems~\cite{braid}.

In this article, we use  a new approach to determine coherent presentations of
monoids. Namely, we use the notion of decreasing
diagrams, that was introduced by van Oostrom in~\cite{oostrom} to obtain a
sufficient condition for a  locally confluent abstract reduction system to 
be confluent. In~\cite{klop} Klop, van Oostrom, and de Vrijer gave a
geometrical proof of the result of van Oostrom. In their proof they discovered
that if an abstract reduction system admits enough 
decreasing elementary diagrams, then every reduction diagram can be patched by 
these elementary diagrams. Their result allows us to give a sufficient
condition for $\left\langle X,\rel, \ed \right\rangle$ to be a coherent
presentation of a monoid $\left\langle X,\rel \right\rangle$ (see
Theorem~\ref{main:presentation}). 
Using 
this sufficient condition we establish a coherent presentation of the $0$-Hecke
monoid. This presentation was used in the joint work  of the
author with A. P.
Santana~\cite{0hecke} to exhibit an action of the  \mbox{$0$-Hecke} monoid on the category of rational
modules for the quantum Borel group. 
Note that our coherent presentation contains the  coherent presentation of the
braid group obtained by Guiraud 
\emph{et al.} in~\cite{braid}. 

The paper is organised as follows. In Section~\ref{ARS} we collect results on
abstract reduction systems used throughout  the paper. In
Section~\ref{cat} we establish  a relationship between abstract reduction systems
and categories, and show how decreasing diagrams can be used to deduce
the commutativity of an infinite set of diagrams, from the commutativity of a given,
often finite, set of diagrams (Theorem~\ref{main:theory}). Section~\ref{action}
contains the  definition of
an action of a monoid on a category (following~\cite{deligne}). 
In Section~\ref{monoids}, we apply Theorem~\ref{main:theory} to an abstract
reduction system associated to a presentation of a monoid. The main result of
Section~\ref{coherent} is Theorem~\ref{main:presentation}, that gives a
sufficient condition for $\left\langle X,\rel,\ed \right\rangle$ to be a coherent presentation. In Section~\ref{hecke} we
describe a coherent presentation of the $0$-Hecke monoid $\rs{n+1}$ of the
symmetric group $\Sigma_{n+1}$.

\section{Abstract reduction system}
\label{ARS}

An \emph{abstract reduction system} is a set $A$ with a relation $R\subset
A\times A$ on $A$ which is called a set of \emph{rewriting rules}. The elements
of $R$ will be sometimes depicted by $a\to b$ for $\left( a,b \right)\in R$. 
The sequence of elements $a_0$, \dots, $a_k$ is called  a \emph{reduction path}
from $a_0$ to $a_k$ if 
$(a_{i-1},a_{i})\in R$ for all $1\le i\le k $. 
If there is a reduction path from $a\in A$ to $b\in A$, we write $a\path b$.

A \emph{reduction diagram} for $\left\langle A,R \right\rangle$ is an oriented planar graph $\Gamma$, such that:
\begin{enumerate}[(1)]
		\setlength{\itemsep}{-2pt}
	\item  All the arrows of $\Gamma$ go either from left to right
		or from top to bottom.  
		\item Some arrows of $\Gamma$ are solid and some arrows of
			$\Gamma$ are dashed. 
	\item  The nodes of $\Gamma$ are labeled by elements of $A$. 
		The label of a node $x$ will be denoted by $l\left( x
		\right)\in A$.
\item If the nodes with labels $a$ and $b$ are connected by a solid arrow then
		${(a,b)\in R}$.
	\item If two nodes are connected by a dashed arrow then they have 
		equal labels. 
	\item  \label{tricky}If from a node $x\in \Gamma$
		there is a horizontal arrow to $y\in \Gamma$ and a vertical arrow to
		$z\in \Gamma$ then one of the two mutually exclusive
possibilities holds
		\begin{enumerate}[(a)]
			\item There is no vertex which is simultaneously  strictly bellow and
strictly to the right of $x$.  		 In this case
		we say that $x$ is an \emph{open corner}. 
	\item There is a node $w\in
		\Gamma$, a vertical path from $y$ to $w$ in $\Gamma$ and a
		horizontal path from $z$ to $w$ in $\Gamma$. These paths are
		called \emph{convergence} paths. 
		If $x$ and $y$ are connected by a dashed arrow, then $z$ and
		$w$ are connected by a dashed arrow as well, and the path from
		$y$ to $w$ contains just one arrow. Similarly, if $x$ and
		$z$ are connected by a dashed arrow then $y$ and $w$ are
		connected by a dashed arrow and the path from $z$ to $w$
		contains just one arrow. 
		\end{enumerate}
		  \end{enumerate}
Bellow is an example of a reduction diagram
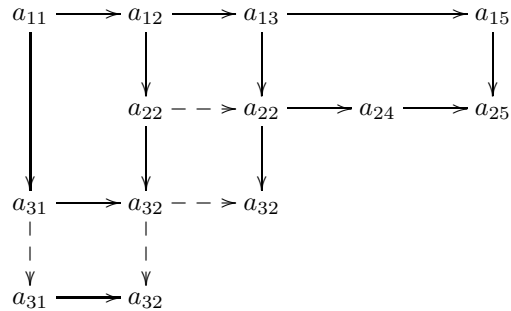
\begin{equation*}
\xymatrix{a_{11} \ar[r] \ar[dd] & a_{12} \ar[r] \ar[d] & a_{13} \ar[rr] \ar[d] &
& a_{15} \ar[d] \\ & a_{22} \ar@{-->}[r] \ar[d]  & a_{22} \ar[r] \ar[d] &
 a_{24}\ar[r] & a_{25}\\
a_{31} \ar[r] \ar@{-->}[d] & a_{32} \ar@{-->}[r] \ar@{-->}[d] & a_{32} \\ a_{31}
\ar[r] &
a_{32} &&&& .}
\end{equation*}
In the above reduction diagram there are two open corners with labels
$a_{22}$ and $a_{32}$, whose (matrix) coordinates are $(2,3)$ and $(3,2)$,
respectively. 
The following graph does not satisfy the  axioms of a reduction diagram
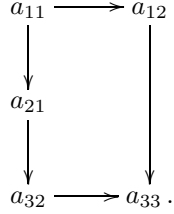
\begin{equation*}
\xymatrix{
a_{11} \ar[r] \ar[d] & a_{12} \ar[dd] \\
a_{21} \ar[d] \\
a_{32} \ar[r] & a_{33}\, .
} 
\end{equation*}
In fact, the top left corner of the above graph is neither an open corner nor
there are convergence paths for the horizontal arrow $a_{11}\to a_{12}$ and the
vertical arrow $a_{11}\to a_{21} $.

  \begin{definition}
		  A reduction diagram is called \emph{complete} if it does not
	  		  contain any open corners. \end{definition}
An example of a complete diagram is given by 
$$
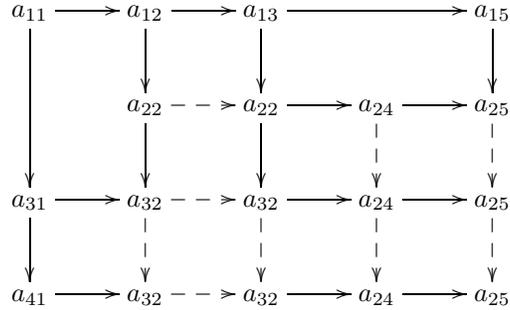
\xymatrix{a_{11} \ar[r] \ar[dd] & a_{12} \ar[r] \ar[d] & a_{13} \ar[rr] \ar[d] &
& a_{15} \ar[d] \\ & a_{22} \ar@{-->}[r] \ar[d]  & a_{22} \ar[r] \ar[d] &
a_{24}\ar[r] \ar@{-->}[d]& a_{25}\ar@{-->}[d]\\
a_{31} \ar[r] \ar[d] & a_{32} \ar@{-->}[r] \ar@{-->}[d] &
a_{32}\ar[r]\ar@{-->}[d] &
a_{24} \ar[r] \ar@{-->}[d] & a_{25} \ar@{-->}[d]  \\ a_{41} \ar[r] &
a_{32}\ar@{-->}[r] &a_{32}\ar[r] & a_{24}\ar[r] & a_{25} }
$$
		  \begin{definition}
			  An \emph{elementary diagram} (e.d.) for $\left\langle A,R
			  \right\rangle$ is a reduction diagram $\Gamma$ such
			  that the edges of $\Gamma$ constitute the boundary of a
			  rectangular. 
		  \end{definition}
		  Suppose $\Gamma$ is an e.d. Then the top side and the left 
		  side of $\Gamma$ contain just one arrow each, as otherwise the top-left
		  corner of $\Gamma$ would not satisfy the axiom \eqref{tricky}
for a reduction diagram.
Therefore there
		  are four different types of e.d.s:
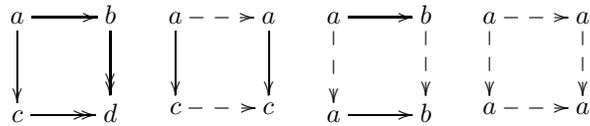
		  $$
		  \xymatrix{a \ar[r] \ar[d] & b \ar@{->>}[d] \\ c \ar@{->>}[r] &
		  d}\ \ \ \  
		  \xymatrix{a \ar@{-->}[r] \ar[d] & a \ar[d] \\ c \ar@{-->}[r] &
		  c}\ \ \ \ 
		  \xymatrix{a \ar[r] \ar@{-->}[d] & b \ar@{-->}[d] \\ a \ar[r] &
		  b}\ \ \ \  
		  \xymatrix{a \ar@{-->}[r] \ar@{-->}[d] & a \ar@{-->}[d] \\ a \ar@{-->}[r] &
		  a} 
		  $$
		  where two headed arrows are used as an abbreviation of
		  a path. 
		  The e.d.s of first type are called \emph{proper} and the
		  rest of e.d.s are called \emph{improper}. 

		  Let $\Gamma$ be a non-complete reduction diagram for
		 $\left\langle A,R \right\rangle$ and $x$ is an open corner in
		 $\Gamma$ with the horizontal arrow to $y$ and  the vertical arrow to
		 $z$.  Suppose that we have an e.d. $\Gamma'$  with the labels $l\left( x
		 \right)$, $l\left( y \right)$, $l\left( z \right)$ at the top
		 left, top right, and bottom left corners, respectively. Then we
		 can glue (suitably stretched) $\Gamma'$ into $\Gamma$ identifying the top left 
		 corner of $\Gamma'$ with $x$, the top right  corner of $\Gamma'$
		 with $y$, and the bottom left  corner of $\Gamma'$ with $z$.  
		 This process is called \emph{adjoining} of the  e.d. $\Gamma'$ to
		 $\Gamma$ at $x$.
It is obvious, that the resulting diagram is again a reduction diagram. 

		 \begin{definition}
			 We say that $\left\langle A,R \right\rangle$ is 
			 \emph{locally confluent} if for any $a\to b$,
			 $a\to c$ there is an e.d. with labels $a$, $b$, $c$ at the
			 top left, top right, and bottom left corners,
			 respectively.
		 \end{definition}
		 
\begin{definition}
	A reduction diagram $\Gamma$ for $\left\langle A,R \right\rangle$ is
	called \emph{initial} if its edges constitute the top and the left sides of a
	rectangular. 
\end{definition}

Suppose that $\left\langle A,R \right\rangle$ is a locally confluent ARS. Let
$\ed$ be a family of e.d.s such that  for every ordered pair $a\to b$, $a\to c$ in $R$
there is an e.d. $E\in \ed$ whose top arrow is $a\to b$ and left arrow
is $a\to c$.  In this case we say that $\ed$ is a complete set of e.d.s. 
 It is
proved in Section~4 of \cite{klop} that the recursive process of adjoining of 
e.d.s from a complete set of e.d.s   to any initial finite diagram $\Gamma$ results in a  complete
diagram $\Gamma'$ in at most  a countable number of   steps.

Recall that a \emph{preorder}  is a reflexive and transitive binary
relation.
Suppose now that the set $R$ is equipped with a preorder $\succeq$. 
We write $r_1 \succ r_2$ if $r_1 \succeq r_2$ but not $r_2 \succeq r_1$. If
$r_1 \succeq r_2$  and $r_2\succeq r_1$ simultaneously, then we write $r_1 \sim r_2$. It is
immediate that  $\sim$ is an equivalence relation on $R$. 
\begin{definition}
We say that the e.d.  
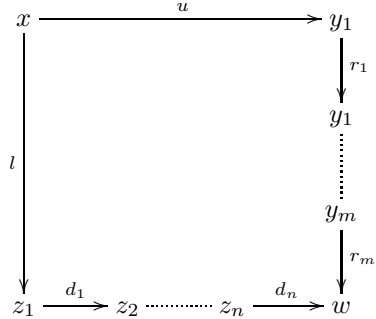
\begin{equation*}
\xymatrix{
x \ar[rrr]^{u} \ar[ddd]_{l} &&& y_1 \ar[d]^{r_1} \\
&  && y_1 \ar@{.}[d] \\
&&& y_m \ar[d]^{r_m} \\
z_1 \ar[r]^{d_1} & z_2 \ar@{.}[r] & z_n \ar[r]^{d_n} & w
}
\end{equation*}
is \emph{decreasing} if the following two condition hold
\begin{multicols}{2}
\begin{enumerate}[1)]
\item there is $0\le j \le n$ such that 
\begin{enumerate}[i)]
\item $u \sim d_j$ in the case $j\not= 0$;
\item $l \succ d_k$ for all $k<j$;
\item $l \succ d_k$ or $u \succ d_k$ for all $k >j$;
\end{enumerate}
\item there is $0\le s \le m$ such that 
\begin{enumerate}[i)]
\item $l \sim r_s$ in the case  $s\not=0$;
\item $ u \succ r_t$ for all $t<s$;
\item $u \succ r_t$ or $l \succ r_t$ for all $t>s$.
\end{enumerate}
\end{enumerate}
\end{multicols}
More informally we require that either  the reduction path $r_1$, \dots, $r_m$ consists
of the steps that are strictly less than $l$ or $u$, or if it starts with the rules
that are strictly less than $u$, then there is a step $r_t$ which is equivalent
to $l$ and all other steps are strictly less than $u$ or $l$. 

Similarly the reduction path $d_1$, \dots, $d_n$ either consists
of the steps that are strictly less than $l$ or $u$, or if it starts with the rules
that are strictly less than $l$, then there is a step $d_t$ which is equivalent
to $u$ and all other steps are strictly less than $u$ or $l$. 
\end{definition}
We say that the preorder $\succeq$ is \emph{well-founded} if for every  sequence
\begin{equation*}
r_1 \succeq r_2 \succeq \dots \succeq r_m \succeq \dots
\end{equation*}
of elements in $R$ there is an integer  $n$ such that for all  $N\ge n$ we have $r_N \sim
r_n$; in other words, any decreasing sequence in $(R,\succeq)$ stabilizes. 

The following theorem is a reformulation of \cite[Proposition~15]{klop}. 
\begin{theorem}\label{geometric}
Suppose $\left\langle A,R \right\rangle$ is a locally confluent ARS and there is
a well-founded preorder $\succeq$ on $R$ and a complete set $\ed$ of decreasing
e.d.s.   Then every process of adjoining chosen
e.d.s to any initial finite diagram $\Gamma$ results in a complete
diagram $\Gamma'$ in a finite number of steps. 
\end{theorem}
We will say that the sequence of elements
\begin{equation*}
a_0, a_1, \dots, a_k
\end{equation*}
is a zigzag in $\left\langle A,R \right\rangle$ from $a_0$ to $a_k$, if for ever $1\le j\le k$, we have
$(a_{j-1},a_j)\in R$ or $(a_j, a_{j-1}) \in R$. We will denote zigzags by $a_0 \zigzag
a_k$. 
Let $\left\langle A,R \right\rangle$ be an ARS, $\succ$ a well-founded preorder
on $R$,  and $\ed$ a complete set of decreasing e.d.s. 
Suppose we are given a reduction diagram of the form
\begin{equation}
\label{zigzag}
\xymatrix{
&& a_1 \ar@{->>}[r] \ar@{->>}[d] & b_0 \\
&& b_1 \ar@{.}[ld]\\
a_k \ar@{->>}[r] \ar@{->>}[d] & b_{k-1}\\
b_k
}
\end{equation}
Then applying Theorem~\ref{geometric} to the diagrams
\begin{equation*}
\xymatrix{
a_j \ar@{->>}[r] \ar@{->>}[d] & b_{j-1} \\
b_j  , 
}
\end{equation*}
we can get in a finite number of steps the diagram
\begin{equation}
\label{zigzag2}
\xymatrix{
&&& a_1 \ar@{->>}[r] \ar@{->>}[d] & b_0 \ar@{->>}[d]  \\
&&& b_1 \ar@{.}[ld] \ar@{->>}[r] & c_0\\
& a_{k-2} \ar@{->>}[r] \ar@{->>}[d] & b_{k-2} \ar@{->>}[d] \\
a_k \ar@{->>}[r] \ar@{->>}[d] & b_{k-1} \ar@{->>}[d] \ar@{->>}[r] & c_{k-2} \\
b_k \ar@{->>}[r]  & c_{k-1} ,
}
\end{equation}
where every square is in fact tiled into an elementary diagrams from $\ed$. Note
that the new diagram has $k-1$ open corners, that is one open corner less than~\eqref{zigzag}.  Now
we can apply Theorem~\ref{geometric} to the diagrams
\begin{equation*}
\xymatrix{
b_j \ar@{->>}[r] \ar@{->>}[d] & c_{j-1} \\
c_j&  . 
}
\end{equation*}
As a result we get a new reduction diagram with $k-2$ open corners. Continuing,
we get a reduction diagram
\begin{equation}
\label{completezz}
\xymatrix{
&& a_1 \ar@{->>}[r] \ar@{->>}[d] & b_0  \ar@{->>}[ddd]\\
&& b_1 \ar@{.}[ld]\\
a_k \ar@{->>}[r] \ar@{->>}[d] & b_{k-1}\\
b_k \ar@{->>}[rrr] &&& z, 
}
\end{equation}
whose interior is tiled by e.d.s from $\ed$. 
Thus we get
\begin{corollary}
\label{cor:zigzag}
Suppose $\left\langle A,R \right\rangle$
is an ARS, $\succeq$ a well-founded preorder on $R$, and $\ed$ a complete set of
decreasing e.d.s. Then any reduction diagram of the form \eqref{zigzag} can be
completed to a diagram of the form \eqref{completezz} in a finite number of steps. 
\end{corollary}

One of the applications of Proposition~\ref{geometric} is to prove
\emph{confluency} of
an ARS. Similarly, Corollary~\ref{cor:zigzag} can be used to prove the Church-Rosser
property of $\left\langle A,R \right\rangle$. In this paper we will apply
Proposition~\ref{geometric} and Corollary~\ref{cor:zigzag} to prove
commutativity of certain diagrams. It should be noted that
Proposition~\ref{geometric} enlightens the long time observed connection between
confluent ARSs and coherence results in category theory.

\section{ARSs and categories}
\label{cat}
We will start by introducing notion that generalises normal forms for
terminating ARS to the case when $\path$ is a well-founded but not necessarily
terminating relation on $A$.
\begin{definition}
We say that $a\in A$ is  \emph{semi-normal} if for every path $a\path b$ in
$\left\langle A,R \right\rangle$ there is a path $b\path a$ in $\left\langle
A,R \right\rangle$. If $c\in A$ and there is a path $c\path a$ to a semi-normal
element $a$, we say that $a$ is a \emph{semi-normal form} of $c$.  
\end{definition}

We will denote
by $\bidir$ the equivalence relation on $A$ defined by
\begin{equation*}
a \bidir b \Leftrightarrow a \path b\,\, \&\,\, b \path a.
\end{equation*}

Suppose the preorder $\path$ on $A$ is well-founded. Then for every element
of $A$ there is a semi-normal form. 

Suppose $\left\langle A,R \right\rangle$ is a confluent system and $b_1$, $b_2$ are two semi-normal forms of $a\in A$. Since
$\left\langle A,R \right\rangle$ is confluent, there is $c\in A$ and two paths
$b_1 \path c$, $b_2 \path c$. As $b_2$ is semi-normal there is  a path
$c\path b_2$. Therefore we get the path $b_1\path c \path b_2$. Since also $b_1$ is
semi-normal,   there is a path $b_2\path b_1$. Hence $b_1 \bidir
b_2$. 
Thus we get
\begin{proposition}
\label{seminormal}
Suppose $\left\langle A,R \right\rangle$ is a confluent ARS, and $\path$ is a
well-founded preorder on $A$. Then for every element $a\in A$ there is a
semi-normal form $b$ which is unique up to equivalence with respect to $\bidir$. 
\end{proposition}
\begin{remark}
Proposition~\ref{seminormal} is a generalization of the well-known fact that
every element in a  confluent terminating ARS has a unique normal form. 
\end{remark}
We will call the set $\attr(a)$ of all semi-normal forms of $a\in A$ an
\emph{attractor} of $a$. 
We also will denote by $\attr(A)$ the set of all semi-normal elements in
$A$. 
Since $a\in \attr(A)$ and $(a,b)\in R$ imply that $b\in \attr(A)$, we see that
the relation $R$ can be restricted to $\attr(A)$. We will denote the resulting
relation on $\attr(A)$ by $\attr(R)$. We will sometimes denote the ARS
$\left\langle \attr(A), \attr(R) \right\rangle$ by $\attr(\left\langle A,R
\right\rangle)$.

Now, let $\cat$ be a category and $\left\langle A,R \right\rangle$ an ARS. Both $\left\langle A,R \right\rangle$ and $\cat$ can
be considered as graphs. Suppose $f\colon \left\langle A,R \right\rangle\to
\cat$ is  a map of graphs. 
 Then using the composition of morphisms in $\cat$, we can extend $f$ to the paths 
$a\path b$
  in $\left\langle A,R \right\rangle$. In particular, given an empty path
$a\dashrightarrow a$, we set $f(a\dashrightarrow a) = 1_{f(a)}$. 
Our aim is to find sufficient conditions on $f$ that guarantee that for any two
paths $p$, $q\colon a\path b$ one gets $f(p) = f(q)$. 
\begin{theorem}\label{main:theory}
Let $\left\langle A,R \right\rangle$ be an ARS, $\succeq$ be a well-founded preorder on $R$ and $\ed$ a complete set of
decreasing e.d.s.  Suppose that
\begin{enumerate}[i)]
\item for every $E\in \ed$ the diagram $f(E)$ is commutative;
\item for every $b\in \attr(A)$ and every path $p\colon b\path b$ the map $f(p)$
is equal to~$1_{f(b)}$. 
\end{enumerate}
Then for any two paths $p$, $q\colon a\path b$ with $b\in \attr(a)$, we get $f(p) = f(q)$. 
\end{theorem}
\begin{proof}
By Theorem~\ref{geometric} we can construct a finite complete reduction
	diagram $\Gamma$ in
	$\left\langle A,R \right\rangle$ whose upper side is $p$,  left side
	is $q$ and which is tiled by e.d.s in $\ed$. It follows that $f\left( \Gamma \right)$ is a
	commutative diagram in $\cat$. 
Suppose that the label of the bottom right corner of $\Gamma$ is $c$. Then
the right  side of $\Gamma$ gives a path
$p'\colon b\path c$ and the bottom side of $\Gamma$ gives a path $q'\colon b\path
c$. From the commutativity of $\Gamma$ we get
\begin{equation}
\label{eq:ilike}
f\left( p' \right) f\left(
p \right) = f\left( q' \right)f\left( q \right).
\end{equation}
Now, since $b$ is semi-normal, there is a path $t\colon c\path b$. By the theorem
assumptions we have $f(tp') = 1_{f(b)} = f(tq')$. 
Thus from \eqref{eq:ilike}, we get
\begin{equation*}
f(p) = f(tp') f(p)=  f(t)f(p')f(p) = f(t)f(q')f(q) = f(tq')f(q) =f(q). 
\end{equation*}
\end{proof}
We will get several corollaries of Theorem~\ref{main:theory}.
\begin{corollary}\label{cor:mono}
Let $\left\langle A,R \right\rangle$ be an ARS, $\succeq$ be a well-founded preorder on $R$ and $\ed$ a complete set of
decreasing e.d.s.  Suppose that
\begin{enumerate}[i)]
\item $\path$ is a well-founded preorder on $A$;
\item for every $E\in \ed$ the diagram $f(E)$ is commutative;
\item for every $r\in R$ the map $f(r)$ is a monomorphism; 
\item for every $b\in \attr(A)$ and every path $p\colon b\path b$ the map $f(p)$
is equal to~$1_{f(b)}$.
\end{enumerate}
Then for any two paths $p$, $q\colon a\path b$ in $\left\langle A,R
\right\rangle$, we get $f(p) = f(q)$. 
\end{corollary}
\begin{proof}
Note that $\attr(b)$ is non-empty since $\path$ is a well-founded preorder. 
Let $b'\in \attr(b) = \attr(a)$.  Then there is a path $s\colon b \path b'$. 
We get two composed paths $sp$, $sq\colon a \path b'$.  We
have $f(sp) = f(s)f(p)$ and  $f(sq) = f(s)f(q)$. By Theorem~\ref{main:theory} we
obtain $f(sp) = f(sq)$. Since $f(s)$ is a monomorphism
$f(s)f(p) = f(s)f(q)$ implies that $f(p)=f(q)$. 
\end{proof}
If $f\colon \left\langle A,R \right\rangle \to \cat$
is such that the map $f(r)$ is invertible for every
$r\in R$, then we can
define, in the obvious way, 
a morphism $f(z)\colon f(a) \to f(b)$ for every zigzag $z\colon a\zigzag b$. 

\begin{corollary}
\label{cor:iso}
Let $\left\langle A,R \right\rangle$ be an ARS, $\succeq$ be a well-founded preorder on $R$ and $\ed$ a complete set of
decreasing e.d.s.  Suppose that
\begin{enumerate}[i)]
\item $\path$ is a well-founded preorder on $A$;
\item for every $E\in \ed$ the diagram $f(E)$ is commutative;
\item for every $r\in R$ the map $f(r)$ is an isomorphism; 
\item for every $b\in \attr(A)$ and every path $p\colon b\path b$ the map $f(p)$
is equal to $1_{f(b)}$.
\end{enumerate}
Then for any two zigzags $z$, $z'\colon a\zigzag b$, we get $f(z) = f(z')$. 
\end{corollary}
\begin{proof}
The zigzags $z$ and $z'$ fit in the following reduction diagram
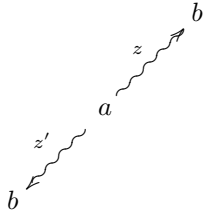
\begin{equation*}
\xymatrix{
&& b\\
& a \ar@{~>}[ur]^{z} \ar@{~>}[dl]_{z'} \\ b
}
\end{equation*}
of type~\eqref{zigzag}. Now, by Corollary~\ref{cor:zigzag}, there is a diagram
of type 
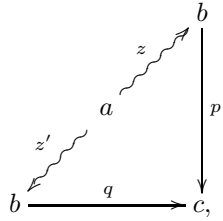
\begin{equation*}
\xymatrix{
&& b \ar[dd]^{p} \\
& a \ar@{~>}[ur]^{z} \ar@{~>}[dl]_{z'} \\ b \ar[rr]^{q} && c,
}
\end{equation*}
which is tiled by diagrams in $\ed$. 
Note, that by Corollary~\ref{cor:mono}, we have $f(p) = f(q)$. 
Since every diagram $f(E)$, $E\in \ed$, is commutative and every map $f(r)$,
$r\in R$, is an isomorphism, we get that 
\begin{equation*}
f(z) f(z')^{-1} = f(p)^{-1} f(q) = 1_{f(b)}. 
\end{equation*}
Therefore $f(z) = f(z')$. 
\end{proof}
\begin{example}
\label{terminating}
Suppose $\left\langle A,R \right\rangle$ is a terminating locally confluent ARS. 
By \cite[Corollary~4.4]{oostrom}
all elementary diagrams for a terminating ARS can be made decreasing.
Further, for every $b\in \attr(A)$ the only path $b\path b$ is the empty one.
Thus we have $f(b\path b) = 1_{f(b)}$ for any map of graphs $f\colon
\left\langle A,R \right\rangle \to \cat$.  
Hence if $f $ is such that
$f(r)$ is an isomorphism for all $r\in R$ and $f(E)$ is commutative for all
$E$ in a complete set $\ed$ of e.d.s, then $f(z)=f(z')$ for any two zigzags
$z$, $z'\colon a\zigzag b$ in $\left\langle A,R \right\rangle$. 
\end{example}

\section{Actions of monoids on categories}
\label{action}

Let $\cat$ be a category and $M$ a monoid with  neutral element $e$. Following
\cite{deligne} we define a \emph{(pseudo)action}
$(\fcal,\lambda)$ of $M$ on $\cat$ as a collection of
\begin{enumerate}[i)]
\item   endofunctors $F_a\colon \cat \to \cat$, $a\in M$,
such that  $F_e \cong Id $ via the natural isomorphism $\eta$;
\item natural isomorphisms $\lambda_{a,b}\colon F_{a}F_b\to F_{ab}$,
such that for all $a$, $b$, $c\in M$ the diagram
\begin{equation*}
\xymatrix{ F_aF_bF_c \ar[r]^-{\lambda_{a,b}F_c} \ar[d]_-{F_a \lambda_{b,c}} &
F_{ab}F_c\ar[d]^-{\lambda_{ab,c}} \\
F_a F_{bc} \ar[r]^-{\lambda_{a,bc}} & F_{abc}}
\end{equation*}
commutes, and $\lambda_{e,a}$, $ \lambda_{a,e}$ 
are induced by $\eta$. 
\end{enumerate}
Given an action $(F,\lambda)$ on $\cat$ and a sequence of elements $a_1$, \dots,
$a_k$, with $k\ge 3$,  we will define recursively the natural isomorphism $\lambda_{a_1,\dots,a_k}$  from 
$F_{a_1}\dots F_{a_k}$ to $F_{a_1\dots a_k}$ by 
\begin{equation*}
\lambda_{a_1,\dots,a_k} = \lambda_{a_1,a_2 \cdots a_k} \circ
F_{a_1} (\lambda_{a_2,\dots,a_k}).\end{equation*}

The actions of $M$ on $\cat$ form a category $[M;\cat]$, where
a morphism from $(F,\lambda)$ to $(F',\lambda')$ is given by a collection of
natural transformations $\rho_a\colon F_a \to F'_a$, $a\in M$, such that the
diagrams
\begin{equation*}
\xymatrix{
F_a F_b \ar[r]^{\lambda_{a,b}} \ar[d]_{\rho_a \rho_b} &  F_{ab}
\ar[d]^{\rho_{ab}} \\
F'_a F'_b \ar[r]^{\lambda'_{a,b}} & F'_{ab}
}
\quad
\xymatrix{
F_e \ar[r]^{\eta} \ar[d]_{\rho_e} & \Id\\ F'_e \ar[ru]_{\eta'} & 
}
\end{equation*}
commute.
From a more abstract point of view, the category of actions of $M$ on $\cat$ is
the category of pseudofunctors from the category $(*,M)$ 
to the $2$-category $\mathrm{Cat}$ of categories.

Suppose $\left\langle X,\rel \right\rangle$ is a presentation of $M$. 
Given $(F,\lambda) \in [M;\cat]$, we get a collection of functors $F_x\colon
\cat\to \cat$ and natural isomorphisms 
\begin{equation*}
\tau_r \colon F_{a_1} \dots F_{a_k} \xrightarrow{\lambda_{a_1,\dots, a_k}}
F_{a_1\dots a_k} = F_{b_1\dots b_l}
\xrightarrow{\lambda_{b_1,\dots,b_l}^{-1}} F_{b_1}\dots F_{b_l}
\end{equation*}
for every relation $r=(a_1\dots a_k,b_1\dots b_l)$ in $\rel$. 
Let us denote by $[X,r; \cat]$ the category whose objects are pairs 
$( (F_x)_{x\in X},(\tau_r)_{r\in \rel})$ 
where $F_x$ are endofunctors of $\cat$ and, for every $r=(a_1\dots a_k, b_1\dots,
b_l)$, $\tau_r$ is a natural isomorphism from $F_{a_1}\dots F_{a_k}$ to
$F_{b_1}\dots F_{b_l}$.  
The  morphisms from $(F,\tau)$ to
$(F',\tau')$ in $[X,r;\cat]$  are  families of  natural transformations $\rho_x\colon F_x\to
F'_x$, $x\in X$, such that for all $r=(a_1\dots a_k,b_1\dots b_l)\in \rel$ the diagrams
\begin{equation*}
\xymatrix{
 F_{a_1} \dots F_{a_k} \ar[r]^{\tau_r} \ar[d]_{\rho_{a_1}\dots \rho_{a_k}} &
F_{b_1}\dots F_{b_l} \ar[d]^{\rho_{b_1}\dots \rho_{b_l}}
\\
  F'_{a_1} \dots F'_{a_k}  \ar[r]^{\tau'_r} & F'_{b_1}\dots F'_{b_l}
}
\end{equation*}
are commutative.
Then we get from the construction described above the restriction functor
$\res\colon [M;\cat] \to [X,\rel;\cat]$. 
\begin{theorem}\label{trm:subcat}
The functor $\res\colon [M;\cat] \to [X,\rel;\cat]$ is full and faithful.
\end{theorem}
Theorem~\ref{trm:subcat} should be well-known. 
In Section~6, we reobtain it as a consequence of Corollary~\ref{cor:iso}. 

%We will show later on that it is
%obvious in view of Corollary~\ref{cor:iso}.

It is clear that it is easier to specify objects in $[X,\rel;\cat]$ than
 objects in $[M;\cat]$. Therefore it is important to have a description of the
essential image of the functor $\res$. This can be done using the coherent
presentations of $M$ described in the next section.  

\section{Monoids and ARS}
\label{monoids}
Let $M$ be a monoid with neutral element $e$ and $\left\langle X,\rel
\right\rangle$ a presentation of $M$. 
We will denote by   $X^*$  the set of all finite words over the alphabet $X$. 
The set $X^*$ will be considered  as a free monoid with multiplication given by
 concatenation of words and  neutral element given by the empty word
$\varnothing$. 
Denote by $\phi$ the canonical epimorphism from $X^*$ to $M$. 
 Let
\begin{equation*}
X^* \rel X^*:= \left\{\, (w_1w w_2, w_1 w'w_2) \,\middle|\, w_1,w_2\in X^*,\, (w,w')\in \rel
\right\}\subset X^* \times X^*. 
\end{equation*} 
We will sometimes write the elements of $X^*\rel X^*$ in the form $w_1 r w_2$
with $r\in \rel$. 
Let us consider the ARS $\left\langle X^*,X^* \rel X^* 
\right\rangle$. 
It is clear that if $w_1 \path w_2$ in $\left\langle X^*,X^*\rel X^*
\right\rangle$, then $\phi(w_1) = \phi(w_2)$. 
\begin{proposition}
Suppose $\left\langle X^*,X^*\rel X^* \right\rangle$ is confluent and $\path$
is a well-founded preorder. Then $\phi(u) = \phi(v)$ if and only if
$\attr(u) = \attr(v)$.
\end{proposition}
\begin{proof}
The ``if'' part is obvious. Suppose $\phi(u) = \phi(v)$. Then there is a
sequence of words
\begin{equation*}
u=w_0, w_1, \dots, w_k = v
\end{equation*}
such  that $(w_j, w_{j-1}) \in X^* \rel X^*$ or $(w_{j-1},w_j) \in X^*\rel X^*$
for all $1\le j\le k$. In other words we have a zigzag $u\zigzag v$ in
$\left\langle X^*, X^* \rel X^* \right\rangle$. Using the fact that $\left\langle X^*,
X^*\rel X^*
\right\rangle$ is confluent and following the same reasoning as in the
proof of Corollary~\ref{cor:zigzag}, we conclude that there are $w \in X^*$ and
two
paths $u\path w$, $v\path w$ in $\left\langle X^*, X^*\rel \right\rangle$. Thus
\begin{equation*}
\attr(u) = \attr(w) = \attr(v). 
\end{equation*} 
\end{proof}
We will denote by $l(w)$ the length of $w\in X^*$. 
If $r = (u,v) \in \rel$, then we define $s(r) = u$ and $t(r)=v$. 
\begin{definition}
\label{criticalpair}
A \emph{critical pair} is a 
pair of elements in $X^* \rel X^*$ of one of the forms
\begin{enumerate}[i)]
\item  $(ur,r'v)$ with $us(r) = s(r')v$ in $X^*$;
\item $(r, ur'v)$ with $s(r) = u s(r') v$. 
\end{enumerate}
We say that a critical pair of the first type is \emph{convergent} if there is
$w\in X^*$ and there are paths $ut(r) \path w$, $t(r')v \path w$. A critical
pair of the second type is called \emph{convergent} if there is $w\in X^*$ and
there are paths $t(r) \path w$, $ut(r')v\path w$ in $\left\langle X^*, X^* \rel
X^*
\right\rangle$. 
\end{definition}
Given convergent critical pairs $(ur,r'v)$, $(r,ur'v)$ and convergence paths
$ut(r) \path w$, $t(r')v\path w$, $t(r) \path w'$, $ut(r')v\path w'$, we
define the following  e.d.s in $\left\langle X^*,X^*\rel X^* \right\rangle$
\begin{equation}
\label{criticaled}
\begin{aligned}
\xymatrix{
{\phantom{s(r)u=}  s(r')v = us(r)} \ar[r]^-{ur} \ar[d]_{r'v} & u t(r) \ar@{->>}[d] \\
 t(r')v \ar@{->>}[r] & w
}& 
\xymatrix{
{\phantom{s(r)=} us(r')v = s(r)} \ar[d]_{ur'v} \ar[r]^-{r} & t(r)
\ar@{->>}[d]\\
ut(r')v \ar@{->>}[r] & w'.
}
\end{aligned}
\end{equation}
We will call the  e.d.s \eqref{criticaled} a  \emph{critical e.d.s}.
Let $\ced$ be a set of critical e.d.s. We say that $\ced$ is complete
if for every critical pair there is at least one corresponding  critical e.d. in
$\ced$.

 Let $r$, $r'\in \rel$ and $w \in W$. We will denote by $\ned(r,w,r')$ the e.d. \begin{equation*}
\xymatrix{
s(r) w s(r') \ar[r]^{rws(r')} \ar[d]_{s(r)w r'} & t(r)ws(r')\ar[d]^{t(r)w r'}\\
s(r) w t(r') \ar[r]^{rwt(r')}  & t(r) wt(r').
}
\end{equation*}
The e.d. $\ned(r,w,r')$ is called a \emph{natural e.d.}
 We write $\ned$ for the set of all natural e.d.s.
Given an e.d. 
\begin{equation*}
E :=
\begin{gathered}
 \xymatrix{
u \ar[r] \ar[d] & v \ar@{->>}[d]  \\
w \ar@{->>}[r] & z
}
\end{gathered}
\end{equation*}
 in $\left\langle X^*, X^*\rel X^* \right\rangle$ and words
$w_1$, $w_2\in X^*$, we define \begin{equation*}
w_1 E w_2 =
\begin{gathered}
 \xymatrix{
w_1 uw_2 \ar[r] \ar[d] & w_1 v w_2\ar@{->>}[d]  \\
w_1 w w_2\ar@{->>}[r] & w_1 z w_2
}
\end{gathered},
\quad  \quad 
E^t =
\begin{gathered}
\xymatrix{
u \ar[r] \ar[d] & w \ar@{->>}[d]  \\
v \ar@{->>}[r] & z
}
\end{gathered}.
\end{equation*} 
Let $\ced$ be a complete set of critical e.d.s. Then the set 
\begin{equation}
\label{eds}
\ed = X^* \ned
X^* \sqcup X^* \ned^t X^* \sqcup X^* \ced X^* \sqcup X^* \ced^t X^*
\end{equation}
 is a complete set of e.d.s for the ARS $\left\langle
X^*, X^* \rel X^* \right\rangle$.

We say that a preorder $\succeq$ on $X^* \rel X^*$ is \emph{monomial} if 
\begin{enumerate}[1)]
\item for every
$\rho_1$, $\rho_2\in X^*\rel X^*$ such that 
$\rho_1 \succ \rho_2$ and for every
$w\in X^*$, we have
\begin{equation*}
w \rho_1 \succ w \rho_2, \quad \rho_1 w \succ \rho_2 w;
\end{equation*}
\item for every
$\rho_1$, $\rho_2\in X^*\rel X^*$ such that 
$\rho_1 \sim \rho_2$ and for every
$w\in X^*$, we have
\begin{equation*}
w \rho_1 \sim w \rho_2, \quad \rho_1 w \sim \rho_2 w.
\end{equation*}
\end{enumerate}
Let us state for the late use 
\begin{proposition}
\label{decreasinged}
Let $\left\langle X,\rel \right\rangle$ be a presentation of a monoid $M$,
$\ced$ a complete set of critical e.d.s and $\ned$ the set of all natural
e.d.s. Define $\ed$  as in~\eqref{eds}.
Suppose $\succeq$ is a monomial preorder on $X^* \rel X^*$ such that all the
e.d.s in $\ced$ and $\ned$ are decreasing. Then all e.d.s in $\ed$ are 
decreasing as well. 
\end{proposition}

\section{Coherent presentation}
\label{coherent}
Let $\left\langle X,\rel \right\rangle$ be a presentation of a monoid $M$. 
Suppose $(F,\tau)\in\left[ X,\rel;\cat \right]$.
Denote by $\End(\cat)$ the category of endofunctors of $\cat$. 
 Define the map of graphs
$f_{F,\tau} \colon \left\langle X^*,X^*\rel X^* \right\rangle \to \End(\cat)$ by
\begin{equation*}
\begin{aligned}
f_{F,\tau} (x_1\dots x_k) &= F_{x_1} \dots F_{x_k}, \quad x_i\in X\\
f_{F,\tau} (x_1 \dots x_k r y_1 \dots y_l)& =  F_{x_1} \dots F_{x_k} \tau_r
 F_{y_1} \dots F_{y_l}, \quad x_i,y_j \in X,\ r\in\rel.
\end{aligned}
\end{equation*}
We will also use the following  natural abbreviations
\begin{equation*}
\begin{aligned}
F_w &:= f_{F,\tau}(w),\ w\in X^*; &  \tau_{\rho} := f_{F,\tau} (\rho), \rho\in
X^* \rel X^*
\end{aligned}
\end{equation*}
 For every path $p$
\begin{equation*}
w_0 \xrightarrow{r_1} w_1 \rightarrow \dots \xrightarrow{r_k} w_k
\end{equation*}
in $\left\langle X^*, X^*\rel X^* \right\rangle$ we denote by $\tau_p$ the
natural isomorphism $\tau_{r_k} \dots \tau_{r_1}$. 
If $p\colon w \path w$ is of length zero then $\tau_p$ is the identity
transformation of $F_w$. Now, given a zigzag $z$
\begin{equation*}
w_0 \stackrel{p_1}{\path} w_1 \stackrel{p_2}{\htap} w_2 \path \dots
w_{k-1} \stackrel{p_{2k}}{\htap} w_{2k},
\end{equation*}
where some paths could be empty, 
we define $\tau_z$ to be the product
\begin{equation*}
\tau_{p_1} \tau_{p_2}^{-1} \tau_{p_3} \dots \tau_{p_{2k-1}}^{-1}. 
\end{equation*}

We say that two zigzags $z_1$ and $z_2$ are parallel if they have the same
source and target. 
Given a set $\zed$ of pairs of parallel zigzags,  we will define $\left[ X,\rel,\zed;\cat \right]$
to be the full subcategory of $[X,\rel ; \cat]$ with objects  $(F,\tau)$ such
that $\tau_p=\tau_q$ for all $(p,q)\in \zed$. 

In what follows, if $E$ is an e.d. then $E\in \zed$ means that the two paths,
that one can obtain from $E$, constitute a pair in $\zed$. Thus if $E\in
\zed$ and $(F,\tau)\in [X,\rel,\zed;\cat]$, then $f_{F,\tau}(E)$ is a
commutative diagram.  

Note that if $E$ is a natural e.d. then $f_{F,\tau}(E)$ is commutative since
$\tau_r$ are natural transformations. Similarly, if $E$ is an e.d. such that
$f_{F,\tau}(E)$ is commutative, then also $f_{F,\tau}(u E v)$ and
$f_{F,\tau}(u E^t v)$  are commutative for all $u$, $v\in X^* $. 
Thus if $\ced$ is a set of critical e.d.s and $\ed$ is defined as
in~\eqref{eds}
then
\begin{equation}\label{ced=ed}
[X,\rel,\ced\sqcup \zed;\cat] = [X,\rel,\ed \sqcup \zed;\cat] 
\end{equation}
for any collection $\zed$ of pairs of parallel zigzags. 
\begin{example}
\label{full}
To avoid ambiguity we write elements in $M^k$ as $m_1|m_2|\cdots|m_k$. Then
we have a
presentation $\left\langle M,\widetilde{\rel} \right\rangle$ of $M$, with
\begin{equation*}
\widetilde{\rel} = \left\{\, \left(m_1| m_2, m_1 m_2\right) \,:\, m_1,m_2\in M \right\}
\sqcup
\{(e,\varnothing)\}.
\end{equation*}
Every path in the resulting ARS $\left\langle M^*,M^*\widetilde{\rel} M^* \right\rangle$
ends either at $m\in M^1$, $m\not=e$, or at $\varnothing$. 
Thus $\left\langle M^*,M^*\widetilde{\rel} M^* \right\rangle$ is terminating. 
Every critical e.d. for $\left\langle M,\widetilde{\rel} \right\rangle$  has one of the following forms
\begin{equation*}
\xymatrix{
m_1|m_2|m_3 \ar[r] \ar[d] & m_1|m_2m_3\ar[d] \\
m_1m_2 | m_3 \ar[r] & m_1m_2m_3
}\
\xymatrix{
m|e \ar[r]^{(m|e,m)} \ar[d]_{m|(e,\varnothing)} & m \ar@{-->}[d]\\
m \ar@{-->}[r]& m
}\
\xymatrix{
e|m \ar[r]^{(e|m,m)} \ar[d]_{(e,\varnothing)|m} & m\ar@{-->}[d] \\
m \ar@{-->}[r] & m.
}
\end{equation*}
Let us denote by $\ced$ the set of all critical e.d.s. Then, by comparing
definitions, we get $[M,\widetilde{\rel},\ced;\cat] = [M;\cat]$. 
Let $\ed$ be defined by~\eqref{eds}. Then, using~\eqref{ced=ed}, we get
\begin{equation*}
[M,\widetilde{\rel},\ed;\cat] = [M;\cat].
\end{equation*}
Let $(F,\lambda) \in [M,\widetilde{\rel},\ed;\cat]$.  By
Example~\ref{terminating}, it follows that for any two zigzags $z$,$z'\colon m_1|\cdots|m_k \to m'_1|\cdots|m'_l$
in $\left\langle M^*,M^*\widetilde{\rel} M^* \right\rangle$ one has  $f_{F,\lambda}(z) =
f_{F,\lambda}(z')$. 
\end{example}

\begin{proposition}
\label{factorize}
Let $\left\langle X,\rel \right\rangle$ be a presentation of a monoid $M$. 
 For any set $\zed$ of pairs of parallel zigzags in $\left\langle X^*,X^*\rel X^* \right\rangle$ the restriction functor $\res$ defined in
Section~\ref{action} factors via the embedding $[ X,\rel,\zed;\cat] \to
[X,\rel;\cat]$. We will denote the resulting functor $[M;\cat]\to
[X,\rel,\zed;\cat]$ by $\res_\paths$. 
\end{proposition}
\begin{proof}
Let $(F,\lambda) \in [M;\cat]$ and $\res(F,\lambda) = (F,\tau)\in [X,\rel,\ed;\cat]$. 
By the definition of $\tau$, every natural isomorphism $\tau_r$, $r\in \rel$, is a
value of $f_{F,\lambda}$ on a suitable zigzag in the defined above ARS
$\left\langle M^*, M^*\widetilde{\rel}M^*  \right\rangle$. 
Now let 
$(p,q) \in \zed$, with $p$, $q\colon u\zigzag w$. We have to prove that $\tau_p = \tau_q$.
Since every $\tau_r$ can be replaced by $\lambda_z$ for a suitable zigzag
$z$ in $\left\langle M^*, M^* \widetilde{\rel} M^* \right\rangle$, we see that
there are zigzags $z'$, $z'' \colon u \zigzag w$ in  $\left\langle M^*, M^* \widetilde{\rel} M^* \right\rangle$
such that $\tau_p = \lambda_{z'}$ and $\tau_q = \lambda_{z''}$. But by
Example~\ref{full}, we have $\lambda_{z'} = \lambda_{z''}$. Thus also $\tau_p =
\tau_q$. 
\end{proof}
\begin{proof}[Proof of Theorem~\ref{trm:subcat}]
The functor $\res$ is faithful as every morphism $\rho \colon (F,\lambda) \to
(G,\mu)$ in $[M;\cat]$ can be uniquely reconstructed from  its image $\nu = \res(\rho)$ by
use of  the diagrams
\begin{equation}
\label{diag}
\begin{gathered}
\xymatrix@C5em@R4em{
F_{x_1} \dots F_{x_k} \ar[r]^{\lambda_{x_1,\dots,x_k}} \ar[d]_{\nu_{x_1}\dots
\nu_{x_k}} & F_{x_1\dots x_k} \ar[d]^{\rho_{x_1\dots x_k}} \\
G_{x_1} \dots G_{x_k} \ar[r]^{\mu_{x_1,\dots,x_k}} 
 & G_{x_1\dots x_k}.
}
\end{gathered}
\end{equation}
Now we show that the functor $\res$ is full. 
Denote the images of $(F,\lambda)$ and of $(G,\mu)$ under $\res$ by $(F,\tau)$
and $(G,\sigma)$, respectively. Take a morphism $\nu\colon (F,\tau)\to
(G,\sigma)$ in $[X,\rel;\cat]$. Let $m\in M$. Then we can write $m$ as a product
of elements in $X$, say $m = x_1\dots x_k$. We define $\rho_m$ by using the
diagram~\eqref{diag}
\begin{equation*}
\rho_m = \mu_{x_1\dots x_k} \circ \nu_{x_1} \dots \nu_{x_k} \circ \lambda_{x_1\dots
x_k}^{-1}.
\end{equation*}
 We have to check that the natural transformation  $\rho_m$ is well-defined. 
Suppose $y_1\dots y_l = m $  with $y_j\in X$. As $\left\langle X,\rel
\right\rangle$ is a presentation of $M$, there is a zigzag $z\colon x_1\dots x_k
\to y_1\dots y_l$ in $\left\langle X^*, X^* \rel X^* \right\rangle$.  
Since $\nu$ is a morphism in $[X,\rel;\cat]$ we get that the diagram
\begin{equation*}
\xymatrix{
F_{x_1} \dots F_{x_k} \ar[r]^{\tau_{z}} \ar[d]_{\nu_{x_1}\dots
\nu_{x_k}} &
 F_{y_1} \dots F_{y_l}\ar[d]^{\nu_{y_1}\dots
\nu_{y_l}} \\
G_{x_1} \dots G_{x_k} \ar[r]^{\sigma_{z}} 
 & G_{y_1} \dots G_{y_l}
}
\end{equation*}
commutes.
As in the proof of Proposition~\ref{factorize}, we can find a zigzag $z'$ in
$\left\langle M^*, M^* \widetilde{\rel}M^* \right\rangle$ such that
$\lambda_{z'} = \tau_z$ and $\mu_{z'} = \sigma_z$. 
Now, consider the zigzag
\begin{equation*}
z'' \colon x_1|\dots | x_k \path m \twoheadleftarrow y_1|\dots|y_l
\end{equation*}
in $\left\langle M^*,M^*\widetilde{\rel}M^* \right\rangle$. Since $z'$ and
$z''$ have the same source and target, we get as in the proof of
Proposition~\ref{factorize}, that $\lambda_{z'} = \lambda_{z''}$ and
$\mu_{z'}= \mu_{z''}$. 
Thus we have the commutative diagram
\begin{equation*}
\xymatrix{
F_{y_1} \dots F_{y_l} \ar[r]^{\lambda_{z''} }\ar@/^6ex/[rr]^{\lambda_{y_1,\dots
,y_l}}
\ar[d]_{\nu_{y_1}\dots
\nu_{y_k}}
 & F_{x_1} \dots F_{x_k} \ar[r]^{\lambda_{x_1,\dots,x_k}} \ar[d]_{\nu_{x_1}\dots
\nu_{x_k}} & F_{x_1\dots x_k} \ar[d]^{\rho_{m}} \\
 G_{y_1} \dots G_{y_l}\ar[r]^{\mu_{z''} }\ar@/_6ex/[rr]_{\mu_{y_1, \dots
,y_l}} &
 G_{x_1} \dots G_{x_k} \ar[r]^{\mu_{x_1,\dots,x_k}} 
 & G_{x_1\dots x_k},
}
\end{equation*}
which shows that $\rho_m$ does not depend on the choice of the presentation of
$m$ in~$\left\langle X,\rel \right\rangle$.  

Now we have to check that $(\rho_m)_{m\in M}$ is a well-defined morphism in
$[M;\cat]$,i.e. that  for every $m_1$, $m_2\in M$ the diagram
\begin{equation*}
\xymatrix{
F_{m_1} F_{m_2} \ar[r]^{\lambda_{m_1,m_2}} \ar[d]_{\rho_{m_1} \rho_{m_2}} &
F_{m_1m_2} \ar[d]^{\rho_{m_1m_2}} \\
G_{m_1} G_{m_2} \ar[r]^{\mu_{m_1,m_2}}  &
G_{m_1m_2}.
}
\end{equation*}
commutes. Suppose $m_1 = x_1\dots x_k$ and $m_2 =
y_1\dots y_l$ with $x_i$, $y_j\in X$. Then since we can use the presentation
$x_1\dots x_k y_1\dots y_l$ of $m_1m_2$ to define $\rho_{m_1m_2}$, we get that
in the diagram 
\begin{equation*}
\xymatrix@C7em{
 F_{x_1} \dots F_{x_k}F_{y_1} \dots F_{y_l} \ar@/^6ex/[rr]^-{
\lambda_{x_1,\dots,x_k,y_1,\dots
,y_l}}\ar[r]^-{\lambda_{x_1,\dots,x_k}\lambda_{y_1,\dots,
,y_l}}
\ar[d]_{\nu_{x_1}\dots \nu_{x_k} \nu_{y_1} \dots \nu_{y_l}}
& F_{m_1} F_{m_2} \ar[r]^{\lambda_{m_1,m_2}} \ar[d]_{\rho_{m_1} \rho_{m_2}} &
F_{m_1m_2} \ar[d]^{\rho_{m_1m_2}} \\
 G_{x_1} \dots G_{x_k}G_{y_1} \dots G_{y_l} \ar@/_6ex/[rr]_-{
\mu_{x_1,\dots,x_k,y_1,\dots
,y_l}}\ar[r]^-{\mu_{x_1,\dots,x_k}\mu_{y_1,\dots,
,y_l}}
& G_{m_1} G_{m_2} \ar[r]^{\mu_{m_1,m_2}}  &
G_{m_1m_2},
}
\end{equation*}
the triangles commute by the definition of the $\lambda$'s, the left rectangle
commutes by the definition of $\rho_{m_1}$ and $\rho_{m_2}$, and the external
rectangle commutes by the definition of $\rho_{m_1m_2}$. Since all $\lambda$'s are
isomorphisms, we get that also the right rectangle commutes. 
This shows that $\rho$ is a well-defined morphism from $(F,\lambda)$ to
$(G,\mu)$ in $[M;\cat]$. 
\end{proof}
We say that a set $\zed$ of pairs of parallel zigzags in $\left\langle X^*,X^*\rel X^*
\right\rangle$ defines a \emph{coherent presentation} of $M$, if the functor
$\res_{\zed}\colon [M;\cat] \to [X,\rel,\zed;\cat]$ is an equivalence of
categories. 
The main result of this section is the following theorem.
%\begin{theorem}\label{main:presentation}
%Let $\left\langle X,\rel \right\rangle$ be a presentation of a monoid $M$ with
%neutral element~$e$. Suppose that $\path$ induced by $X^* \rel X^*$ is
%a well-founded preorder on $X^*$, and that we can found a well-founded preorder $\succeq$ on $X^* \rel
%X^*$ and a set of critical e.d.s $\ced$ such that
%all diagrams in $\ced$ and all natural e.d.s are decreasing. Denote by
%$\paths$ the collection of all pairs $(p,\varnothing_b)$, where
%$b\in\attr(A)$, $\varnothing_b$ is the empty path at $b$, and $p\colon b\path
%b$. Then $\res_{\ced\sqcup \paths}$ is an equivalence of categories. 
%\end{theorem}
\begin{theorem}\label{main:presentation}
Let $\left\langle X,\rel \right\rangle$ be a presentation of a monoid $M$ with
neutral element~$e$. Suppose that 
\begin{enumerate}[1)]
\item the transitive closure $\path$ of $X^* \rel X^*$ is
a well-founded preorder on $X^*$;
\item  there exists a well-founded monomial preorder $\succeq$ on $X^* \rel
X^*$ such that all natural e.d.s are decreasing;
\item there is a complete set $\ced$ of critical e.d.s that are decreasing with
respect to~$\succeq$.
\end{enumerate}
 Denote by
$\loops$ the collection of all pairs $(p,\varnothing_b)$, where
$b\in\attr(A)$, $\varnothing_b$ is the empty path at $b$, and $p\colon b\path
b$. Then $\res_{\ced\sqcup \loops}$ is an equivalence of categories. 
\end{theorem}
\begin{proof}
Let $\ed$ be the collection of e.d.s defined by \eqref{eds}.  In view of \eqref{ced=ed}, it is enough to
show that
\begin{equation*}
\res_{\ed\sqcup \loops}\colon [M;\cat] \to [X,\rel, \ed\sqcup \loops; \cat]
\end{equation*}
is an equivalence of categories. 
Since $\res_{\ed\sqcup \loops}$ is fully faithful, it is enough to check that it
is a dense functor. 
Let $(F,\tau)\in [X,\rel,\ed\sqcup \loops ;\cat]$. Then
\begin{equation*}
f_{F,\tau} \colon \left\langle X^*,X^* \rel X^* \right\rangle \to \End(\cat)
\end{equation*}
satisfies conditions of Corollary~\ref{cor:iso}. 
In particular, for any two zigzags $z$, $z'\colon u\zigzag v$ in $\left\langle X^*, X^* \rel X^*
\right\rangle$ we have $\tau_z = \tau_{z'}$. 

Now, for every $m$ in $M$ we choose a presentation $x_1\dots x_k$ of $m$ in
$\left\langle X,\rel \right\rangle$. We will assume that if $m\in X$, then its
chosen presentation is $m$ itself. 
We define $G_m = F_{x_1}\dots F_{x_k}$ using the above chosen presentation.
Now, let $m'\in M$ and $y_1\dots y_l$ be its chosen presentation. Then
$x_1\dots x_k y_1\dots y_l$ is a presentation of $m m'$, which, in general, is 
different of the chosen presentation, say, $z_1\dots z_n$ of $mm'$. 
Since $\left\langle X,\rel \right\rangle$ is a presentation of $M$, we get that
there is zigzag $\zeta\colon x_1\dots x_k y_1 \dots y_l \zigzag z_1 \dots z_n$
in $\left\langle X^*, X^* \rel X^* \right\rangle$. We define $\lambda_{m,m'} =
\tau_{\zeta}$. As we already mentioned the resulting natural isomorphism is independent
of the choice of $\zeta$.
Now, if $m''\in M$, the compositions
\begin{equation}
\begin{aligned}
\label{comp}
& G_m G_{m'} G_{m''} \xrightarrow{G_m \lambda_{m',m''}} G_m G_{m'm''}
\xrightarrow{\lambda_{m,m'm''}} G_{mm'm''}\\
& G_m G_{m'} G_{m''} \xrightarrow{ \lambda_{m,m'}G_{m''}}  G_{mm'}G_{m''}
\xrightarrow{\lambda_{mm',m''}} G_{mm'm''}
\end{aligned}
\end{equation}
are equal to $\tau_{\zeta'}$ and $\tau_{\zeta''}$ for suitable parallel zigzags
$\zeta'$, $\zeta''$ in
$\left\langle X*,X^*\rel X^* \right\rangle$.
Thus the natural transformations \eqref{comp} are equal.
Hence we get that  $(G,\lambda)$ is an object of $[M;\cat]$. 

We have to check that $\res(G,\lambda) = (F,\tau)$. Since for every $x\in X$, we
have $G_x = F_x$, we get $\res(G,\lambda) = (F,\sigma)$. Thus, we have only to
check that $\sigma = \tau$. For every $r = (u,v)\in \rel$, we defined $\sigma_r$ as
$\lambda_z$ for  a suitable zigzag $z$ in $\left\langle M^*, M^*
\widetilde{\rel} M^* \right\rangle$. Since every $\lambda_{m,m'}$ is of the form
$\tau_{z'}$ for a zigzag $z'$ in $\left\langle X^*, X^* \rel X^* \right\rangle$,
we get that there is a zigzag $\zeta\colon u \zigzag v$ in $\left\langle X^*,
X^*\rel X^*
\right\rangle$ such that $\lambda_z = \tau_\zeta$. But now $\tau_\zeta =
\tau_r$. This shows that $\sigma_r = \tau_r$.  
\end{proof} 
\begin{example}
\label{knuth}
Suppose $\left\langle X^*, X^* \rel X^* \right\rangle$ is terminating and
$\rel$ satisfies the Knuth-Bendix condition. Let $\ced$ be a complete set of
critical e.d.s. Then, in view of Example~\ref{terminating}, we get that
$\res_\ced $  is an equivalence of categories. Thus one of the ways to find a
coherent presentation of a monoid $M$ is to perform the Knuth-Bendix completion
procedure on a given presentation. Unfortunately, this process can either not
to finish or to lead to an enormous coherent presentation which is not useful for practical
purposes.  
\end{example}

\section{Coherent presentation for the  $0$-Hecke monoid}
\label{hecke}

The $0$-Hecke monoid $\rs{n+1}$ is defined by the following presentation
\begin{equation*}
\begin{aligned}
& X = \left\{ T_1,\dots, T_n \right\}\\[3ex]
&{\rel'}=\left\{ a_{i}\,\middle|\, 1\le i\le n \right\}\sqcup \left\{\, b_{i+1}
\,\middle|\, 1\le i\le n-1 \right\} \sqcup \left\{\, c_{ji} \,\middle|\,
 1 \le i \le j-2 \le n-2
\right\}, 
\end{aligned}
\end{equation*}
where
\begin{equation*}
a_{i} = (T_iT_i, T_i) ,\quad b_{i+1} = (T_{i+1}T_i T_{i+1}, T_i
T_{i+1}T_i),\quad c_{ji} = (T_{j}T_i , T_i T_j).
\end{equation*}
In this section we find a coherent presentation of $\rs{n+1}$ that extends the
above presentation of $\rs{n+1}$. 
It is easy to see that $\left\langle X^*,X^* {\rel'} X^* \right\rangle$
is not (locally) confluent.
We  denote by $\rel''$ the set $\rel'\sqcup \left\{\, c_{ij} \,\middle|\,
i\le j-2
\right\}$, where $c_{ij} = (T_iT_j, T_j T_i)$. 
It is well-know that $\left\langle X^*, X^* \rel'' X^* \right\rangle$ is
confluent.  

In the diagrams below we will write $i$ in place of $T_i$. 
We will also use ${}'$ for decrement and $\,\widehat{\ }\,$ for increment, thus for an integer
$k$
\begin{equation*}
k' = k-1, \quad \widehat{k} = k+1.
\end{equation*}
We will show that the set $\paths$ of  pairs of paths  in
$\left\langle X^*, X^*\rel'' X^* \right\rangle$ defined below gives a coherent
presentation of $\rs{n+1}$, i.e. that 
\begin{equation*}
\res_{\paths}\colon [\rs{n+1}; \cat]
\to [ X,\rel'', \paths ; \cat] 
\end{equation*}
is an
equivalence of categories. 
The set $\paths$ consists of the pairs of paths that one obtains from the
following diagrams
\begin{equation}
\label{loops}
\begin{gathered}
\xymatrix{
st \ar[r]^{c_{st}} \ar@/_3ex/@{-->}[rr] & ts \ar[r]^{c_{ts}} & st
}
\end{gathered}
\end{equation}

\begin{equation}
\label{aa}
\begin{gathered}
\xymatrix{
kkk \ar[r]^{ka_k} \ar[d]_{a_k k} & kk\ar[d]^{a_k}\\
kk \ar[r]^{a_k} & k
}
\end{gathered}
\end{equation}
\begin{equation}
\label{ba}
\begin{gathered}
\xymatrix@C4em{
kk'kk \ar[rr]^-{kk'a_k} \ar[d]_{b_kk} && kk'k \ar[d]^{b_k} \\
k'kk'k \ar[r]^-{k'b_k} & k'k'kk' \ar[r]^{a_{k'}k k'}  & k'kk'
}
\end{gathered}
\end{equation}
\begin{equation}
\label{ab}
\begin{gathered}
\xymatrix{
kkk'k \ar[r]^{kb_k} \ar[dd]_{a_kk'k} & kk'kk'\ar[d]^{b_kk'} \\
& k'kk'k' \ar[d]^{k'ka_{k'}} \\
kk'k \ar[r]^{b_k} & k'kk'
}
\end{gathered}
\end{equation}

\begin{equation}
\label{ac}
\begin{gathered}
\xymatrix{
stt \ar[rr]^-{sa_t}  \ar[d]_{c_{st}t} && st \ar[d]^{c_{st}}\\
tst \ar[r]^-{tc_{st}} & tt s \ar[r]^{a_t} & ts
}
\end{gathered}
\end{equation}

\begin{equation}
\label{bb}
\begin{gathered}
\xymatrix{
kk'kk'k  \ar[d]_{b_k k'k} \ar[r]^-{kk'b_k} & kk'k'kk' \ar[r]^{ka_{k'} kk'} & kk'kk' \ar[d]^{b_k k'} \\
k'kk'k'k \ar[d]_{k'ka_{k'} k} && k'kk'k' \ar[d]^{k'k a_{k'}}  \\
k'kk'k \ar[r]^-{k'b_k} & k'k'kk' \ar[r]^{a_{k'}k k'} & k'kk'
}
\end{gathered}
\end{equation}

\begin{equation}
\label{bc}
\begin{gathered}
\xymatrix{
tss's \ar[r]^-{tb_s} \ar[d]_{c_{ts} s's} & t s'ss' \ar[r]^-{c_{ts'}ss'} & s'tss'
\ar[d]^{s'c_{ts} s'} \\
sts's \ar[d]_{sc_{ts'} s} && s'sts' \ar[d]^{s's c_{ts'}} \\
ss'ts \ar[r]^{ss'c_{ts} } & ss'st \ar[r]^{b_s t} & s'ss't
}
\end{gathered}
\end{equation}

\begin{equation}
\label{cc}
\begin{gathered}
\xymatrix{
kji \ar[r]^-{kc_{ji}} \ar[d]_{c_{kj}i} &  kij \ar[r]^{c_{ki}j} & ikj
\ar[d]^{ic_{kj}} \\
jki \ar[r]^-{jc_{ki}} & jik \ar[r]^{c_{ji}k} & ijk
}
\end{gathered}
\end{equation}

\begin{equation}
\label{TZ}
\begin{gathered}
\xymatrix@C3.4em{
\widehat{k}kk'\widehat{k}k\widehat{k} \ar[r]^-{\widehat{k}kk'b_{\widehat{k}}} \ar[d]_{\widehat{k}kc_{k'\widehat{k}} k\widehat{k}} & 
\widehat{k}kk'k\widehat{k}k \ar[r]^-{ \widehat{k}b_k \widehat{k}k} & \widehat{k}k'kk' \widehat{k}k \ar[r]^-{c_{\widehat{k}k'} kk'\widehat{k}k} & k'\widehat{k}kk'\widehat{k}k
\ar[d]^{k'\widehat{k}kc_{k'\widehat{k}}k} \\
\widehat{k}k\widehat{k}k'k\widehat{k}  \ar[d]_{b_{\widehat{k}} k'k\widehat{k}} 
&&& k'\widehat{k}k\widehat{k}k'k \ar[d]^{k'b_{\widehat{k}} k'k} \\
k\widehat{k}kk'k\widehat{k} \ar[d]_{k\widehat{k} b_k \widehat{k}} 
&&& k'k\widehat{k}kk'k \ar[d]^{k'k\widehat{k} b_k} \\
k\widehat{k}k'kk'\widehat{k} \ar[d]_{kc_{\widehat{k}k'} kk'\widehat{k}} 
&&& k'k\widehat{k}k'kk' \ar[d]^{k'kc_{\widehat{k}k'} kk'} \\
kk'\widehat{k}kk'\widehat{k} \ar[r]^-{kk'\widehat{k}kc_{k'\widehat{k}}} & kk'\widehat{k}k\widehat{k}k' \ar[r]^-{kk'b_{\widehat{k}}k'} &
kk'k\widehat{k}kk' \ar[r]^-{ b_k \widehat{k}kk'} & k'kk'\widehat{k}kk'.
}
\end{gathered}\,\,\,\,\,\,
\end{equation}
The diagrams \eqref{cc} and \eqref{TZ} were called
\emph{Tits-Zamolodchikov $3$-cells} in~\cite{braid}. 

We will say that a path $p$ in $\left\langle X^*, X^* \rel'' X^* \right\rangle$
is a \emph{$c$-path} if all the steps in $p$ are of the form $X^* c_{st}X^*$ for
some $s$, $t$ such that $|s-t| \ge 2$. 

\begin{proposition}
\label{cpath}
Suppose $(F,\tau) \in \left[ X,\rel'', \paths; \cat \right]$. 
If $p$, $p'\colon u\path v$ are two $c$-paths in $\left\langle X^*, X^*\rel'' X^*
\right\rangle$, then $\tau_{p} = \tau_{p'}$. 
\end{proposition}
\begin{proof}
Let $\relc$ be the subset $\rel''$ consisting of all $c_{ji}$ with $j\ge i+2$.
Using~\eqref{loops}, we see that $\tau_{c_{ij}} = \tau_{c_{ji}}^{-1}$ for all
$j\ge i+2$. Therefore, there are zigzags $z$, $z'$ in $\left\langle X^*, X^*
\relc X^*
\right\rangle$ such that $\tau_{p} = \tau_z$ and $\tau_{p'} = \tau_{z'}$.

The ARS $\left\langle X^*, X^* \relc X^* \right\rangle$ is terminating and
locally confluent, with the only critical e.d.s given by~\eqref{cc} with
$k\ge j+2$, $j\ge i+2$. Therefore, by Example~\ref{terminating}, we get that
$\tau_{z} = \tau_{z'}$. 
\end{proof}
The easiest way to prove our result would be to have a complete set of
decreasing critical e.d.s for $\left\langle X^*, X^*\rel'' X^* \right\rangle$.
The author does not know at the moment if this is possible. Instead, we will
proceed as follows
\begin{enumerate}[1)]
\item first we replace $\rel''$ with a bigger set $\rel$ of generating rules;
\item then we define a preorder on  $X^* \rel X^*$ so that there is a complete
set of decreasing critical e.d.s for $\left\langle X, \rel \right\rangle$ and
all natural e.d.s are decreasing;
\item further we show that all the chosen critical e.d.s can be subdivided into
diagrams in $\paths$;
\item finally we show that for all $(F,\tau)\in [X,\rel, \paths;\cat]$ and any
attractor loop $p\colon b\path b$ one has $\tau_p = 1_{F_b}$. 
\end{enumerate}
%\subsection{Second approach}
%Let 
%\begin{equation*}
%\rel=\rel'' \sqcup \left\{\, D_j \,\middle|\, 3\le j\le n \right\}\sqcup
%\left\{\, C_{ij} \,\middle|\, 2\le i \le j-2\le n-2  \right\}.
%\end{equation*}
%where
%\begin{equation*}
%D_j = (T_j T_{j''}T_{j'} T_j, T_{j''} T_{j'} T_j T_{j'}),\quad C_{ij} =
%(T_iT_{i'}T_j, T_j T_iT_{i'}).
%\end{equation*}
%\subsection{First approach}
Let
\begin{equation*}
\rel = \left\{\, a_{i} \,\middle|\,  1\le i \le n \right\} \sqcup 
\left\{\, b_{ji} \,\middle|\,  1\le i <j \le n \right\} \sqcup \left\{\,
c_{st}
\,\middle|\, |s-t| \ge 2 \right\}.
\end{equation*}
Here $a_i$, $c_{st}$ are defined as before and 
\begin{equation*}
b_{ji} = (j\dots ij,\, j'jj'\dots i), 
\end{equation*}
where $j\dots i$ denotes the interval of the arithmetic progression with the step
$-1$. 
Note that in particular $b_{jj'} = b_j$, and therefore $\rel''\subset \rel$.
We define the preorder $\succeq$ on $X^* \rel X^*$ as follows. Given 
$r$, $r'$, we write $r \gg r'$ to indicate that $urv \succ u'r'v'$ for all
$u$, $v$, $u'$, $v'\in X^*$. Moreover, for example, $c \gg a$ will indicate that
$c_{st} \gg a_i$ for all $s$, $t$, and $i$.  We will write
$\#_j u$ for the number of occurrences of $j$ in $u\in X^*$. 
The preorder $\succeq$ is defined by the rules
\begin{enumerate}[$1)$]
\item  for $i\le j''$, $b_{ji} \sim b_{jj'} j''\dots i$, i.e. to compare two elements in $X^*\rel X^*$,
we first replace, if necessary, all $b_{ji}$ by $b_{jj'}j''\dots i$, and then proceed with
the rules below;
\item we order the relations in $\rel$ by 
\begin{enumerate}[i)]
\item $b \gg c\gg a$;
\item $b_{i+1,i} \gg b_{i,i-1} $;
\item $c_{ij} \gg c_{ts}$ for all $j\ge i+2$ and $t\ge s+2$;
\item for $k\ge j+2$ and $t\ge s+2$, $c_{kj} \gg c_{ts}$, if $k> t$ or
$k = t$ and $j > s $; 
\item  $\left\{\, c_{ij} \,\middle|\,  j\ge i+2 \right\}$ are ordered arbitrary;
\end{enumerate}
\item if words contain  the same generating rule $r\in \rel$ we proceed as
follows:
\begin{enumerate}[i)]
\item if $r=a_i$, then we just compare lengths of the words: the
longer word is greater; 
\item if $r=c_{ji}$ with $j\ge i+2$, then $u c_{ji} v \succ u' c_{ji} v'$ if
\begin{enumerate}[a)]
\item $\sum\limits_{k\ge j} \#_k u > \sum\limits_{k\ge j} \#_k u'$; \vspace{2ex}
\item $\sum\limits_{k\ge j} \#_k u = \sum\limits_{k\ge j} \#_k u'$
and $\sum\limits_{k\le i} \#_k v > \sum\limits_{k\le i} \#_k v'$. \vspace{3ex}
\end{enumerate}
\item if  $r=b_{kk'}$, then  $ub_{kk'} v \succ u'b_{kk'} v'$ 
 if 
\begin{equation*}
(\#_k uv , \#_{k'} u v, \dots, \#_1 u  v) >
(\#_k u'  v' , \#_{k'} u'  v', \dots, \#_1 u' v') 
\end{equation*}
with respect to the lexicographical order on $\N^k$. 
\end{enumerate}
\end{enumerate}
If for two words $w$, $w'\in X^*\rel X^*$ we cannot conclude either $w\succ w'$
or $w'\succ w$, according to the above rules, then $w\sim w'$. 
It is obvious that the preorder $\succeq$ is monomial. 
\begin{proposition}
All the natural e.d.s of the ARS $\left\langle X^*, X^*\rel X^* \right\rangle$
are decreasing with respect to the preorder $\succeq$. 
\end{proposition}
\begin{proof}
Let $r,r'\in \rel$ and $w\in X^*$. Then $\ned(r,w,r')$ has the top arrow
$rws(r')$,  the bottom arrow $rwt(r')$, the left arrow $s(r)wr'$, and the right
arrow $t(r)wr'$. 
If $rws(r')\succeq rwt(r')$ and $s(r)wr'\succeq t(r)wr'$ then the e.d.
$\ned(r,w,r')$ is obviously decreasing. 

Thus, we have to find triples $(r,w,r')$ for which $rwt(r')\succ rws(r')$ or
$t(r)wr' \succ s(r)wr'$. 

First we identify triples $(r,w,r')$ such that $rwt(r') \succ rws(r')$. 
If $r$ is of type $a$, then $rwt(r')\succ rws(r')$ if and only if the length of
$t(r')$ is greater than the length of $s(r')$. But there is no rule $r'\in \rel$
with such property. Thus $r$ cannot be of type $a$. 

Now, suppose $r$ if of type $c$. We have two cases: either $r=c_{ij}$ or
$r=c_{ji}$, with $j\ge i+2$. If $r=c_{ij}$ then $rws(r') \sim rwt(r')$. If
$r=c_{ji}$ then $rwt(r') \succ rws(r')$ if and only if 
\begin{equation}
\label{kir}
\sum_{k\le i} \#_k t(r') > \sum_{k\le i} \#_k s(r'). 
\end{equation}
It is clear that $r'$ can be of type $a$ or $c$ as for such rules we have
$\#_k t(r') = \#_k s(r')$ for arbitrary $k$. If $r'=b_{ml}$ with $m>l$, then 
\begin{equation}
\label{jtrp}
\#_k t(r') = 
\begin{cases}
\#_k s(r') -1, & k= m\\
\#_k s(r') + 1, & j=m-1\\
\#_k s(r'), & \mbox{otherwise.}
\end{cases}
\end{equation}
Thus \eqref{kir} holds only if $i=m-1$, i.e. $r'=b_{i+1,l}$ with $i+1>l $. 
As $b\gg c$ we get that $s(r)wr' \succ rwt(r')$. Moreover, 
$s(r)wr'\sim t(r)wr'$ as $\#_k s(c_{ji}) = \#_k t(c_{ji})$ for all $k$. 
Thus $\ned(r,w,r')$ is decreasing in this case. 

Let us consider the case $r=b_{ji}$ for some $j>i$. Then we have $rwt(r')> rws(r')$ if
and only if 
\begin{equation*}
\left( \#_j t(r'),\dots \#_1 t(r') \right) > (\#_j s(r'),\dots, \#_1 s(r')).
\end{equation*}
This is possible only in the case $r'=b_{j+1,k}$ for some $k<j+1$. 
As $b_{j+1,j} \gg b_{jj'}$, we get that $s(r)wr'\succ rwt(r')$. Moreover
\begin{equation*}
\begin{aligned}
\#_{j+1}s(r) = \#_{j+1}(j\dots ij)=0=\#_{j+1}(j'j\dots i) =
\#_{j+1}t(r)\\
\#_{j}s(r) = \#_j (j\dots i j)  = 2 > 1 = \#_j (j'j\dots i) = \#_j t(r).
\end{aligned}
\end{equation*} 
Thus $s(r)wr'\succ t(r)wr'$ and we get that $\ned(r,w,r')$ is decreasing.

Now we will analyze the triples $(r,w,r')$ such that $t(r)wr' \succ s(r)wr'$. 
It is impossible to have $r'$ of type $a$ as the length of $t(r)w$ never exceeds
the length of $s(r)w$. 
Suppose $r'$ is of type $c$. Then $r'=c_{ij}$ or $r'=c_{ji}$ for some $j$,
$i$ such that $j>i+1$. If $r'=c_{ij}$ then $t(r)wr' \sim s(r) wr'$. Thus we have
only to consider the case $r'=c_{ji}$. In this situation $t(r)wr'\succ
s(r)wr'$ if and only if 
\begin{equation}
\label{kjr}
\sum_{k\ge j }\#_k t(r) > \sum_{k\ge j} \#_k s(r). 
\end{equation}
As for rules $r''$ of type $a$ or $c$, the inequality $\#_k t(r'') \le \#_k
s(r'')$ holds for
every $k$, we get that $r$ cannot  be of type $a$ or $c$. If $r=b_{ml}$ for some
$m>l$, then the sums in \eqref{kjr} are equal unless $j=m$, in which case the
left hand side of \eqref{kjr} is less than the right hand side of
\eqref{kjr}. This shows that $r'$ cannot have type $c$.

Suppose $r'=b_{ji}$ for some $j>i$. Then $t(r)wr' \succ s(r)wr'$ if and only if 
\begin{equation*}
\left( \#_j t(r),\dots, \#_1 t(r) \right) > \left( \#_j s(r),\dots, \#_1
s(r) \right)
\end{equation*}
with respect to the lexicographical order on $\N^j$. This is possible only in
the case $r=b_{j+1,l}$ for some $l<j+1$.  
As $b_{j+1,j} \gg b_{jj'}$, we get that $rws(r') \succ t(r)wr'$. Moreover, as 
$\#_{j+1} s(r') = \#_{j+1} t(r')$ and $\#_j s(r') > \#_j t(r')$, we have
$rws(r') \succ rwt(r')$.  This shows that $\ned(r,w,r')$ is decreasing in this
case. 
\end{proof}

Now we construct a complete set of decreasing critical e.d.s
for the ARS $\left\langle X^*, X^*\rel X^* \right\rangle$. 

Note that there are not any rules $r$, $r'\in \rel$
such that $s(r')$ is a proper subword of $s(r)$. Therefore all the cricital
pairs for $\rel$ are of the form $(ru,vr')$ for some $r$, $r'\in \rel$ and
non-empty $u$, $v\in X^*$. 

 If $r=c_{ij}$ with
$j\ge i+2$, then the following commutative diagram is decreasing and  resolving
for this critical
pair  
\begin{equation*}
\begin{gathered}
\xymatrix{
iju \ar[r]^{c_{ij}u} \ar[dd]_{vr'} & jiu\ar[d]^{c_{ji}u} \\
& iju \ar[d]^{vr'} \\
t(r) \ar@{-->}[r] & t(r).
}
\end{gathered}
\end{equation*} 
In fact, $c_{ij} \gg c_{ji}$ implies that $c_{ij}u \succ c_{ji} u $.  As
$\succeq$ is monomial, we get 
that $vr'\sim vr'$. 
Similarly, if $r'=c_{ij}$ with $j\ge i+2$, the diagram
\begin{equation}
\label{conc}
\begin{gathered}
\xymatrix{
uij \ar[r]^{uc_{ij}} \ar[dd]_{ru} & uji\ar[d]^{uc_{ji}} \\
& uij \ar[d]^{ru} \\
t(r) \ar@{-->}[r] & t(r).
}\end{gathered}
\end{equation}
is decreasing and resolves the critical pair $(ru,vc_{ij})$. 

It is left to consider the critical pairs with $r$ and $r'$ equal to one of 
$a_k$, $b_{ji}$, $j\ge i+1$, $c_{ts}$, $t\ge s+2$. 
In the diagrams below we will abbreviate $c$-paths by~$\stackrel{c^*}{\path}$. This does not create any ambiguity in view of
Proposition~\ref{cpath}. 

We first consider all the critical pairs involving at least one $a_k$. 
The diagram~\eqref{aa} is  a decreasing critical e.d. for the critical pair
$(ka_k, a_kk)$. 
For the critical pair $(a_k k'\dots jk, kb_{kj})$, we consider the diagram

\begin{equation}
\label{aB}
E(a_k,b_{kj})=\begin{gathered}
\xymatrix{
kk\dots jk \ar[r]^-{kb_{kj}} \ar[dd]_{a_k k'\dots jk} & kk'k\dots j \ar[d]^{b_ k'\dots  j} \\
& k'k k'k'\dots j \ar[d]^{k'k a_{k'} k''\dots j}\\
k\dots jk \ar[r]^-{b_{kj}} & k'k\dots j. 
}
\end{gathered}
\end{equation}
The diagram~\eqref{aB} is decreasing since $kb_{kj}$ is greater than all the other
arrows. Note, that for $j=k'$ the word $k''\dots j$ is empty and we
recover~\eqref{ab}. 

For the critical pair $(b_{kj}k, k\dots j a_k)$, we consider the diagram
\begin{equation}
\label{Ba}
E(b_{kj}, a_k) = \begin{gathered}
\xymatrix@C4em{
k\dots jkk \ar[rr]^-{k\dots ja_k} \ar[d]_{b_{kj} k} && k\dots jk\ar[d]^{b_{kj}}\\
k'k\dots j k \ar[r]^-{k'b_{kj}} &  k'k'k\dots j\ar[r]^-{a_{k'} k\dots j} & k'k\dots
j.
}
\end{gathered}
\end{equation}
The diagram~\eqref{Ba} is decreasing since $b_{kj}k$ is greater then any other
arrow. 
In the case $j=k'$, we recover the diagram~\eqref{ba}. 

All the critical pairs involving two $b$-rules are of the form $(b_{kj}k'\dots
i, k\dots jb_{ki})$. For $j=k'$, we get
\begin{equation}
\label{bB}
{}_{E(b_{kk'}, b_{ki}) =} \begin{gathered}
\xymatrix{
kk'k\dots ik \ar[rrr]^-{kk'b_{ki}} \ar[ddd]^{b_{kk'}k'\dots ik} &&& kk'k'k\dots
i \ar[d]^{ka_{k'} k\dots i}\\
&&& kk'k\dots i \ar[d]^{b_{kk'}k'\dots i} \\
&&& k'kk'k'\dots i \ar[d]^{k'ka_{k'} k''\dots i}\\
k'kk'k'\dots ik \ar@/^4ex/[r]^-{k'ka_{k'}k''\dots ik } & k'k\dots ik
\ar[r]^-{k'b_{ki}} & k'k'k\dots i\ar[r]^-{a_{k'}k\dots i} & k'k\dots i.
}
\end{gathered}
\end{equation}
The diagram~\eqref{bB} is decreasing since $b_{kk'}k'\dots ik\sim
kk'b_{ki}$ dominates all
the other arrows. Note that for $i=k'$ the subword $k''\dots i$ is empty and we
recover the diagram~\eqref{bb}. 

The diagram corresponding to  the case $j\le k''$ is
not drawn in the rectangular form for the
typographical
reasons. We also write $S$ for the word $k'''\dots j k''\dots i$. 

\begin{equation}
\label{BB}
{}_{E(b_{kj}, b_{ki})=} \begin{gathered}
 \xymatrix@R8ex@C0.2em{
& k\dots jk\dots ik \ar[ld]_{b_{kj} k'\dots ik} \ar[rd]^{k\dots
jb_{ki}} \\ 
k'k\dots jk'\dots ik \ar@{->>}[d]_{c^*}  & &  k\dots jk'k\dots i
\ar@{->>}[d]^{c^*} \\
k'k k' k'' k' k S \ar[d]_{k'k b_{k'k''} k S} && 
kk' k'' k' k k' S \ar[d]^{kb_{k'k''} kk'S } \\
k'kk'' k'k'' k S  \ar@{->>}[d]_{c^*}&&
kk''k'k'' kk' S \ar@{->>}[d]^{c^*} \\
k'k''kk'kk'' S \ar[d]_{k'k'' b_{kk'} k'' S}  && k''kk'kk'' k' S
\ar[d]^{k''b_{kk'} k'' k' S} 
\\
k'k''k'kk'k'' S  \ar[d]_{b_{k'k''} kk'k''S} && k''k'k k' k'' k' S
\ar[d]^{k'' k'k b_{k'k''}S} \\ 
k''k'k'' kk'k''  S\ar@{->>}[rr]^{c^*} && k'' k' k k'' k' k'' S
}
\end{gathered}
\end{equation}
In~\eqref{BB} all the arrows are dominated by the two top arrows 
$b_{kj} k'\dots ik \sim k\dots jb_{ki}$. For $j=k''$ and $i=k'$,
\eqref{BB} becomes \eqref{TZ}. 

Now we will consider critical pairs involving one $b$-rule and one $c$-rule of
the form $c_{kj}$ with $k\ge j+2$. 

For $(c_{kj}j'\dots ij, k b_{ji})$, $k\ge j+2$, $j\ge i+1$, 
we consider the diagram
\begin{equation}
\label{cB}
{}_{E(c_{kj}, b_{ji})} = \begin{gathered}
\xymatrix{
kj\dots ij \ar[rrr]^-{kb_{ji}} \ar[d]_{c_{kj} j'\dots ij} &&& kj'j\dots i
\ar@{->>}[d]^{c^*} \\
jkj'\dots ij \ar@{->>}[r]^{c^*} & j\dots i k j \ar[r]^{j\dots i c_{kj}} &  j\dots ij k \ar[r]^{b_{ji} k} & j'j\dots ik.
}
\end{gathered}
\end{equation}
In~\eqref{cB}, we have $kb_{ji} \sim b_{ji} k$ and $kb_{ji}$  dominates every
arrow in the vertical $c$-path. Thus we have to show that $c_{kj} j'\dots ij$
dominates every arrow in the horizontal $c$-path. The horizontal $c$-path
involves the rules generated by $c_{ks}$ with $s<j$ and at the last step the
rule $j\dots i c_{kj}$. Since, by definition, $c_{kj} \gg c_{ks}$ if $j>s$, we
see that $c_{kj} j'\dots ij$ dominates all the horizontal $c$-arrows except
probably the last one. But it is easy to see that also
\begin{equation*}
c_{kj} j'\dots ij \succ j\dots i c_{kj}. 
\end{equation*}

Now suppose $k\ge s+2$ and $k\ge j+1$. For the critical pair $(b_{kj} s, k\dots
j c_{ks})$ we will consider several cases. If $j\ge s+2$, then we can take a
diagram similar to~\eqref{cB}
\begin{equation}
\label{Bc1}
E(b_{kj},c_{ks})=\begin{gathered}
\xymatrix{
k\dots jk s \ar[r]^-{k\dots jc_{ks}} \ar[ddd]_{b_{kj} s}  & k\dots jsk
\ar@{->>}[d]^{c^*}\\ 
& ks k'\dots jk \ar[d]^{c_{ks} k'\dots jk} \\
& s k\dots j k \ar[d]^{s b_{kj}}\\
k'k\dots j s \ar@{->>}[r]^-{c^*} & s k'k\dots j. 
}
\end{gathered}
\end{equation}
In~\eqref{Bc1}, $b_{kj}s \sim sb_{kj}$ and dominates every arrow in the
horizontal $c$-path. We have to check that $k\dots j c_{ks}$ dominates all
arrows in the vertical $c$-path. This is done using that $c_{ks} \gg
c_{ts}$ for all $k'\ge t \ge j$, and that $k\dots j c_{ks} \succ c_{ks} k'\dots
jk$. 

For $j=s+1$ we consider the diagram
\begin{equation}
\label{Bc2}
E(b_{kj}, c_{kj'}) = \begin{gathered}
\xymatrix@C6em{
k\dots jkj' \ar[r]^{k\dots j c_{kj'}} \ar[d]_{b_{kj} j'} & k\dots j'k
\ar[d]^{b_{kj'}} \\
k'k\dots j' \ar@{-->}[r] & k'k\dots j'.
}
\end{gathered}
\end{equation}
The diagram~\eqref{Bc2} is decreasing since $b_{kj} j' \sim b_{kk'} k''\dots j'
\sim b_{kj'}$. 

For $j=s$, we consider the diagram
\begin{equation}
\label{Bc3}
E(b_{kj}, c_{kj}) = \begin{gathered}
\xymatrix@C6em{
k\dots jkj \ar[r]^{k\dots j c_{kj}} \ar[dd]_{b_{kj} j} & k\dots jjk
\ar[d]^{k\dots \widehat{\jmath} a_j k} \\
& k\dots jk \ar[d]^{b_{kj}} \\
k'k\dots jj \ar[r]^{k\dots \widehat{\jmath} a_j} & k'k\dots j
}
\end{gathered}
\end{equation}
In~\eqref{Bc3}, $b_{kj}j$ dominates all the arrows. 

For $k'' \ge s \ge j+1$, we take the diagram
\begin{equation}
\label{Bc4}
E(b_{kj}, c_{ks}) = \begin{gathered}
\xymatrix@C6em{
k\dots jks \ar[r]^{k\dots j c_{ks}} \ar[ddd]_{b_{kj} s} & k\dots js k
\ar@{->>}[d]^{c^*} \\
& k\dots \widehat{s} k s\dots js \ar[d]^{k \dots \widehat{s}k b_{sj}} \\
& k \dots \widehat{s}k s's\dots j \ar[d]^{b_{k\widehat{s}} s's \dots j} \\
k'k\dots \widehat{s} s \dots js \ar[r]^-{k'k \dots \widehat{s} b_{sj}} & k'k\dots
\widehat{s} s's \dots j.
}
\end{gathered}
\end{equation}
Note that for $s=k''$ the subword $k''\dots \widehat{s}$ is empty. We claim that
$b_{kj}s$ dominates all the arrows. This is obvious for all arrows except
$b_{k\widehat{s} } s'\dots j$. 
We have $b_{kj} s\sim b_{kk'} k''\dots j s$ and $b_{k\widehat{s}} s's\dots j
\sim b_{kk'} k''\dots \widehat{s} s' s \dots j$. Now for all $s+1 \le t \le k$, we
have
\begin{equation*}
\begin{aligned}
\#_t (k''\dots j s)  & =
\begin{cases}
0 ,& t=k,k'\\
1, & s+1\le t\le k''
\end{cases}
\\[3ex] & = \#_t (k''\dots \widehat{s} s's \dots j).
\end{aligned}
\end{equation*}
Further,
\begin{equation*}
\#_s (k''\dots js ) = 2 > 1 = \#_s (k''\dots \widehat{s} s's \dots j).
\end{equation*}
Thus
\begin{equation*}
b_{kj} s \succ b_{k\widehat{s}}  s's \dots j. 
\end{equation*}
It is left to consider the critical pairs involving two $c$-rules of the form
$c_{ts}$ with $t\ge s+2$. They are all of
the type $(c_{kj}i, kc_{ji})$. The diagram~\eqref{cc} gives a decreasing
convergence diagram for this pair as $c_{kj}i$ dominates all the arrows
in~\eqref{cc}. This is true as $c_{kj} \gg c_{ki} \gg c_{ji}$ and $c_{kj}i >
ic_{kj}$ for $i<j$. 

We will denote the set of pairs of paths that one obtains from
the chosen critical e.d.s
by $\ced$. 
We also write $\loops$ for the set $(p,\varnothing_w)$, were  $p\colon
w\path w$  are the loops at semi-normal elements of $\left\langle X^*, X^* \rel
X^*
\right\rangle$. By Theorem~\ref{main:presentation}, we get that
$\res_{\ced\sqcup \loops}$ is an equivalence of categories.

Now we identify pairs of paths coming from loops in the attractor of
$\left\langle X^*, X^* \rel X^* \right\rangle$. 
\begin{proposition}
The preorder $\path$ on $\left\langle X^*, X^*\rel
X^*
\right\rangle$ is well-founded. Equivalently $\attr(w) \not=\varnothing$
for every $w\in X^*$. 
\end{proposition}
\begin{proof}
Suppose 
\begin{equation}
\label{seq}
w_0 \to w_1 \to \dots \to w_k \to \dots
\end{equation}
is an infinite sequence in $\left\langle X^*, X^*\rel
X^*
\right\rangle$. Then we have 
\begin{equation*}
l(w_0) \ge l(w_1) \ge \dots \ge l(w_k) \ge \dots
\end{equation*}
Since the set of words of length not greater than $l=l(w_0)$ is finite, we see
that there is a word $w$ that appears in~\eqref{seq} infinitely many times. Let
$w_k = w$ be the first appearance of $w$ in~\eqref{seq}. Then for all $m>k$,
there is $n > m$ such that $w_n = w$. Thus we get that for all $m>k$ there are
paths $w = w_k \path w_m$ and $w_m \path w_n = w$. This shows that for every
$m>k$, we have $w \bidir w_m$, that is $\path$ is a  well-founded preorder.
\end{proof}
\begin{proposition}
\label{attractor}
Let $w$ be a semi-normal element in $\left\langle X^*, X^*\rel
X^*
\right\rangle$ and $p\colon w\path w$ a path in $\left\langle X^*, X^*\rel
X^*
\right\rangle$. Then $p$ is a $c$-path.
\end{proposition}
\begin{proof}
 Define the map
\begin{equation*}
\begin{aligned}
\underline{l} \colon X^* & \to \N^n\\
u & \mapsto (l(u), \#_n(u), \#_{n-1}(u), \dots, \#_2(u)). 
\end{aligned}
\end{equation*}
We will write $u \ge v$ if $\underline{l}(u) \ge \underline{l}(v)$ with respect to
the lexicographical order on $\N^n$. If $u\ge v$ and $v\ge u$, then we write
$u \equiv v$. 
It is obvious that $(u,v) \in X^* \rel X^*$ implies that $u\ge v$ and $u\equiv
v$ if and only if $(u,v) = w'c_{ij}w''$ for some $w'$, $w''\in X^*$. Thus
$p$ is a path in $\left\langle X^*, X^*c
X^*
\right\rangle$. 
%Now, we can replace every arrow $w'c_{ji}w''$ with $j\ge i+2$ in
%the path $p$ by the arrow $w'c_{ij}w''$ in the opposite direction. In this way
%we get a zigzag in $\left\langle X^*, X^*c^+
%X^*
%\right\rangle$.
\end{proof}
Let $\loops'$ be a subset of $\loops$ consisting of \eqref{loops}.
Combining Proposition~\ref{cpath} and Proposition~\ref{attractor}, we get
\begin{corollary}
The functor $\res_{\ced \sqcup \loops'}$ is an equivalence of categories. 
\end{corollary}
Next we are going to relate the categories $[X,\rel, \ced \cup \loops'; \cat]$
and $[X,\rel'',\paths; \cat]$, where $\paths$ was defined on
page~\pageref{loops}. 

We have two functors $\res \colon [X,\rel; \cat] \to [X,\rel'';\cat]$ and
$\ext \colon [X,\rel'';\cat]\to [X,\rel; \cat]$. The first functor is defined by $(F,\tau)\to
(F, \mathrm{res}(\tau))$, where $\mathrm{res}(\tau)$ is the restriction of
$\tau$ to $\rel''$.
The functor $\ext \colon [X,\rel'';\cat]\to [X,\rel;\cat]$
is defined by $\ext(F,\tau) = (F,\tilde{\tau})$
where
\begin{equation*}
\begin{aligned}
\tilde{\tau}(a_i) & = \tau(a_i) ,&  \tilde{\tau}(c_{st}) & = \tau(c_{st}), &
\tilde{\tau}(b_{kk'}) = \tau(b_k)
\end{aligned}
\end{equation*}
and $\tilde{\tau}(b_{kj})$ for $j\le k''$ is computed  recursively
using the relation
\begin{equation*}
\tilde{\tau}(b_{kj}) = (\tilde{\tau}(b_{k,j+1})F_j) \circ (F_k\dots F_{j+1}
\tau(c_{jk})).
\end{equation*}

Let $\zed$ be the set of pairs of paths in $\left\langle X^*, X^*\rel X^*
\right\rangle$obtained from~\eqref{Bc2}. Then 
it is easy to see that $\res$ and $\ext$ induce mutually inverse equivalences of
categories
\begin{equation*}
\res \colon [X,\rel,\zed; \cat] \rightleftarrows [X,\rel'';\cat]\colon \ext
\end{equation*}
It is also clear that
if $(F,\tau) \in [X,r,\ced \sqcup \loops'; \cat]$ then $\res(F,\tau) \in
[X,\rel'', \paths;\cat]$ since the images of the pairs in $\paths$ under
$\ext$ appear among the pairs
$\ced \sqcup \loops' $. 

Now we are going to show that for every $(F,\tau) \in [X,\rel'', \paths; \cat]$,
we get $\ext (F,\tau) \in [X,\rel, \ced\sqcup \loops';\cat]$. Let us write
$(F,\tilde{\tau})$ 
for $\ext(F,\tau)$. We have to show that every diagram in $\ced \sqcup \loops'$ is mapped into a
commutative diagram under $f_{F,\tilde{\tau}}$.  For $\loops'$ this is obvious,
since $\loops' \subset \paths$ and $(F,\tau) \in [X,\rel'',\paths;\cat]$. 
Now, we have to check that all the diagrams (\ref{conc}-\ref{Bc4}) become
commutative under $f_{F,\tilde{\tau}}$. 

For~\eqref{conc} this is obvious, since $\tau(c_{st}) \tau(c_{ts}) = \id$ for
all $|s-t| \ge 2$.  
The commutativity of the other diagrams follows by an induction argument and the  
patching diagrams listed bellow. 
We label natural e.d.s by $\ned$. 

As we already noted before~\eqref{aB} for $j=k'$ becomes~\eqref{ab}.
Now, for $j\le k''$, we have
\begin{equation}
\label{aBp}
\begin{gathered}
\xymatrix{ 
kk\dots j k \ar[rr]^-{ kk\dots \widehat{\jmath} c_{jk} }  \ar[dd]_{a_k k'\dots jk } && kk\dots \widehat{\jmath}kj 
\ar[dd]^{a_k k'\dots \widehat{\jmath}kj} \ar[rr]^-{k b_{k\widehat{\jmath}}j} && kk'k\dots \widehat{\jmath}j\ar@{->>}[dd] \\
& \ned && E(a_k,b_{k\widehat{\jmath}})j \\
k\dots jk \ar[rr]^{k\dots \widehat{\jmath}c_{jk}} && k \dots \widehat{\jmath}kj \ar@{->>}[rr] && k'k\dots j.
}
\end{gathered}
\end{equation}

Further~\eqref{Ba} for $j=k'$ is~\eqref{ba}, and for $j\le k''$, we have the
diagram
\begin{equation}
\label{Bap}
\begin{gathered}
\xymatrix{
k\dots jkk\ar[dd]_{k\dots \widehat{\jmath} c_{jk} k} \ar[rrrr]^-{k\dots j a_k} &&&& k\dots jk  \ar[dd]^{k\dots \widehat{\jmath} c_{jk}}\\
&& \eqref{ac}  \\
k\dots \widehat{\jmath} kjk \ar[rr]^-{k\dots \widehat{\jmath} kc_{jk}}
\ar[dd]_{b_{k\widehat{\jmath}}jk}
 && k\dots \widehat{\jmath} kkj \ar[rr]^{k\dots \widehat{\jmath}  a_kj} \ar[dd]_{b_{k\widehat{\jmath}}kj}&& k\dots \widehat{\jmath} kj  \ar@{->>}[dd]\\ 
& \ned && E(b_{k\widehat{\jmath}},a_k) j
\\
k'k \dots j k \ar[rr]^{k'k\dots \widehat{\jmath} c_{jk}} && k'k \dots \widehat{\jmath} kj \ar@{->>}[rr] &&
k'k\dots \widehat{\jmath}j.
}
\end{gathered}
\end{equation}
The diagram~\eqref{bB} for $i=k'$ is \eqref{bb}. For $i\le k''$, we use the
diagram
\begin{equation}
\label{bbp}
\begin{gathered}
\xymatrix@C2em{
kk'k\dots ik\ar[dd]_{b_{kk'}k'\dots ik}  \ar[rr]^{kk'k\dots \widehat{\imath} c_{ik}} && kk'k \dots \widehat{\imath}ki
\ar[rr]^-{kk'b_{k\widehat{\imath}}}\ar[dd]_{b_{kk'} k'\dots \widehat{\imath}ki}  && kk'k' k \dots \widehat{\imath}i
\ar@{->>}[dd] \\
& \ned && E(b_{kk'}, b_{k\widehat{\imath}}) i \\
k'kk'k' \dots ik \ar[rr]^-{k'kk'k' \dots \widehat{\imath} c_{ik}} && k'kk'k' \dots \widehat{\imath}k i
\ar@{->>}[rr] && k'k\dots \widehat{\imath}i.
}
\end{gathered}
\end{equation}

The diagram~\eqref{BB} for $j=k''$ and $i=k'$ is~\eqref{TZ}. Let us first
consider the case $i=k'$. Then we use the patching diagram
\begin{equation}
\label{BBp1}
\begin{gathered}
\xymatrix@C0em{
k\dots j kk'k\ar[dd]_{k\dots \widehat{\jmath} c_{jk} k'k}  \ar[rrrr]^-{k\dots j b_{kk'}} &&&&
k\dots jk'kk' \ar@{->>}[dd]^{c^*} \\
&& \eqref{bc} \\
k \dots \widehat{\jmath} kjk'k \ar[rr]^{c^*} \ar[dd]_{b_{k\widehat{\jmath}} jk'k } && k\dots \widehat{\jmath} kk' kj
\ar[dd]_{b_{k\widehat{\jmath}}k'kj} 
\ar[rr]^{k\dots \widehat{\jmath} b_k j} && k \dots \widehat{\jmath} k'kk' j \ar@{->>}[dd] \\
& \ned && E(b_{k\widehat{\jmath}},b_{kk'}) j\\
k'k\dots \widehat{\jmath}j k' k  \ar[rr]^{c^*} && k'k \dots \widehat{\jmath} k'k j \ar@{->>}[rr] && k'' k' k
k'' k' k'' \dots \widehat{\jmath} j.
}
\end{gathered}
\end{equation}
For $i\le k''$, we use 
\begin{equation}
\label{BBp2}
\begin{gathered}
\xymatrix@C0em{
k\dots j k\dots i k \ar[dd]_{b_{kj} k'\dots ik}  \ar@/^3ex/[rr]^-{k\dots j
k\dots \widehat{\imath} c_{ik}} && k\dots j k\dots \widehat{\imath} ki \ar[dd]_(0.7){b_{kj} k' \dots \widehat{\imath}ki}
\ar[rr]^{ k\dots j b_{k\widehat{\imath}}i} && k\dots j k' k \dots \widehat{\imath} i \ar@{->>}[dd] \\
& \ned && E(b_{kj}, b_{k\widehat{\imath}}) i \\
k'k\dots jk'\dots ik \ar@/_4ex/[rr]_-{k'k\dots j k' \dots \widehat{\imath}c_{ik}} && 
k'k \dots jk' \dots \widehat{\imath}ki \ar@{->>} [rr] && k''k'kk''k'k''\dots jk''\dots \widehat{\imath}i.
}
\end{gathered}
\end{equation}

The diagram~\eqref{cB} for $i=j'$ is \eqref{bc}. For $i\le j''$, we consider the
diagram
\begin{equation}
\label{cBp}
\begin{gathered}
\xymatrix@C0em{
kj\dots ij \ar@/^3ex/[rr]^{kj\dots vc_{ij}}\ar[dd]_{c_{kj} j'\dots ij} &&
kj\dots v ji \ar[rr]^{kb_{jv}i} \ar[dd]^(0.7){c_{kj} j' \dots vji} &&
kj'j\dots vi \ar@{->>}[dd]\\
& \ned && E(c_{kj}, b_{jv}) i \\
jkj' \dots ij \ar@/_4ex/[rr]_-{jkj' \dots v c_{ij}} && jkj' \dots v ji \ar@{->>}[rr] &&
j'j \dots vk i.
}
\end{gathered}
\end{equation}
The commutativity of $f_{F,\tilde\tau}(E(b_{kk'}, c_{ks}))$ for $s\le k'''$ follows
from the commutativity of the diagram obtained by application  $f_{F,\tau}$ to~\eqref{bc}.
Now, we show that $f_{F,\tilde\tau}(E(b_{kj}, c_{ks}))$ is commutative for
$s\le k'''$ and $j\le k''$ by induction on $j$, using the diagram
\begin{equation}
\label{Bc1p}
\begin{gathered}
\xymatrix@C0em{
\ar[dd]_{k\dots \widehat{\jmath} c_{jk} s} k\dots j ks \ar[rrrr]^-{k\dots j c_{ks}} &&&& k\dots j sk\ar[dd]^{c^*} \\
 && \eqref{cc}\\
\ar[dd]_{b_{k\widehat{\jmath}}js} k\dots \widehat{\jmath} kjs \ar@/^3ex/[rr]^{k\dots \widehat{\jmath} kc_{js}}&&  k\dots \widehat{\jmath} ks j
\ar[rr]^{k\dots \widehat{\jmath} c_{ks} j} \ar[dd]_(0.7){b_{k\widehat{\jmath}}sj}  && \ar@{->>}[dd]  k \dots \widehat{\jmath}sk j \\
& \ned && E(b_{k\widehat{\jmath}},c_{ks})j \\
k'k\dots js \ar@/_4ex/[rr]_{k'k\dots \widehat{\jmath} c_{js}} && k'k\dots \widehat{\jmath} sj \ar@{->>}[rr] && sk'k
\dots \widehat{\jmath}j.
}
\end{gathered}
\end{equation}
The commutativity of the diagrams $f_{F,\tilde\tau}(E(b_{kj}, c_{kj'})) $ follows from the definition
of $\tilde\tau$. 
Next we check that $f_{F,\tilde\tau}(E(b_{kj}, c_{kj}))$ is commutative. 
\begin{equation}
\label{Bc3p2}
\begin{gathered}
\xymatrix@C1.2em{
\ar@/_3ex/[dd]_{k\dots v c_{j k} j}^{{\eqref{loops}}} k \dots j k j
\ar[rrrr]^{k\dots j c_{kj}} &&&& k\dots jjk  \ar[dd]^{ k\dots v a_j k}\\ && \eqref{ac}
\\
 \ar@/_3ex/[uu]_(0.3){k\dots v c_{kj}
j}
k \dots v kjj \ar[rr]_{k\dots v ka_j} \ar[dd]_{b_{kv}jj}  && k \dots v kj \ar[rr]^{k\dots v
c_{kj}}\ar[dd]^{b_{kv} j} && k\dots v jk \ar[dd]^{b_{kj}} \\ & \ned && E(b_{kv}, c_{kj}) 
\\
k'k\dots v kjj \ar[rr]^{k'k\dots v k a_j}  && k'k\dots v j  \ar@{-->}[rr] &&
k'k\dots v_j
}
\end{gathered}
\end{equation}

To prove the commutativity of $f_{F,\tilde{\tau}}(E(b_{kj},c_{ks}))$, one notices
that commutativity of  $f_{F,\tilde\tau} (E(b_{kj}, c_{kj'})$ for all $j\le k'$, implies that 
for any $k''\ge s \ge j+1$, the following diagram commutes uppon
application of $f_{F,\tilde\tau}$ to it:
\begin{equation}
\label{asdf}
\begin{gathered}
\xymatrix{
k\dots j k s \ar[r] \ar[r]^-{c^*} \ar[d]_{b_{k j} s} & k\dots \widehat{s}k s
\dots j s \ar[d]^{b_{k\widehat{s}} s \dots js} \\
k'k \dots j s \ar@{-->}[r] & k' k \dots j s .
}
\end{gathered}
\end{equation}
Further one uses that $\tilde\tau (c_{sk})\tilde\tau ( c_{ks}) = \id$ and the
natural e.d. with $r=b_{k\widehat{s}}$, $r' = b_{sj}$, and $w = \varnothing$. 

This finishes the proof that the diagrams (\ref{aa}-\ref{TZ}) give a coherent
presentation for the $0$-Hecke monoid $\rs{n+1}$.

\bibliography{preaction}
\bibliographystyle{amsplain}

\end{document}